\documentclass[11pt]{amsart}
\usepackage{color, amssymb, amsmath,amsfonts,mathrsfs,tikz}
\usetikzlibrary{arrows.meta}
\usepackage{enumitem}

\usepackage[numbers,sort]{natbib}

\usepackage{ mathrsfs, amsmath, amsthm, amssymb,
thm-restate, pifont}
\usepackage{textcomp}

\usepackage[colorlinks]{hyperref}
\hypersetup{linkcolor=blue, urlcolor=blue, citecolor=red}

\usepackage[capitalise]{cleveref}

\newtheorem{theorem}{Theorem}[section]
\newtheorem{proposition}[theorem]{Proposition}

\newtheorem{corollary}[theorem]{Corollary}
\newtheorem{lemma}[theorem]{Lemma}
\newtheorem{question}[theorem]{Question}

\newtheorem{remark}[theorem]{Remark}

\newtheorem{definition}[theorem]{Definition}
\newtheorem{example}[theorem]{Example}

\newcommand{\N}{\mathbb N}

\newcommand{\End}{\operatorname{End}}

\newcommand{\im}{\operatorname{im}}

\newcommand{\makeset}[2]{\left\lbrace #1 \;\middle|\;
 \begin{tabular}{@{}l@{}}
   #2
  \end{tabular}
  \right\rbrace}
  \usepackage{url, amsmath, amsthm, amssymb}
\usepackage{amsmath}
\usepackage{multicol}
\usepackage{hyperref}
\usepackage{multicol}
\usepackage{mathrsfs}
\usepackage{abstract}
\hypersetup{
    colorlinks,
    citecolor=black,
    filecolor=black,
    linkcolor=black,
    urlcolor=black
}
\newcommand{\makepres}[2]{\left\langle #1 \;\middle|\;
 \begin{tabular}{@{}l@{}}
   #2
  \end{tabular}
  \right\rangle}

\newcommand{\mb}[1]{\mathbb{#1}}
\newcommand{\mc}[1]{\mathcal{#1}}

\renewcommand{\a}{\alpha}
\renewcommand{\l}{\lambda}
\newcommand{\tree}{\operatorname{tree}}

\DeclareMathOperator{\pst}{pst}
\DeclareMathOperator{\PEnd}{PEnd}
\DeclareMathOperator{\tot}{\operatorname{tot}}
\DeclareMathOperator{\rt}{root}
\DeclareMathOperator{\shrub}{shrub}

\email{bldw@st-andrews.ac.uk}
\email{luna.elliott142857@gmail.com}

\title[Thompson Monoids And J\'{o}nsson-Tarski Algebras]{Finite Presentability of Brin-Higman-Thompson Monoids via Free Multidimensional J\'{o}nsson-Tarski Algebras}
\author{Bill de Witt, Luna Elliott}

\begin{document}
\maketitle
\begin{abstract}
We show that the monoids \(\operatorname{tot} M_{k,1}\) introduced by Birget \cite{ birget2009monoid, birget2020monoid} and their generalizations \(\operatorname{tot} nM_{k,r}\), which extend the Brin-Higman-Thompson groups, can be realized as the endomorphism monoids of higher-dimensional J\'{o}nsson-Tarski algebras. We also show how elements of these monoids can be thought of as ``rewrite rules".
We use these representations to show that the monoids are finitely presented.
\end{abstract}

\tableofcontents

\section{Introduction}

Thompson's group \(V\) (along with \(T\leq V\)) was the first known example of a finitely presented infinite simple group. Since then, the group \(V\) has seen great attention in the literature. See for example \cite{higman1974finitely,cannon1996introductory, belk2014conjugacy, bleak2016further,  bleak2017infinite}.

A monoid version of \(V\) was first considered by Thompson. Thompson has given talks about this monoid, but to our knowledge, he has not published on the topic. 
Birget describes in \cite{birget2009monoid} several monoid extensions \(M_{k,1}\) of the well-studied groups \(G_{k,1}\) (where \(G_{2,1}=V\)). These extensions involve replacing the bijective ``prefix exchange" maps of \(G_{k,1}\) with partial maps on prefix codes, which need not be injective and may not have incomparable image sets. Birget also introduced the following submonoids: \(\operatorname{sur} M_{k,1}\), using partial functions whose range contains a complete prefix code; \(\operatorname{tot} M_{k,1}\), using non-partial functions; and \(\operatorname{Inv} M_{k,1}\), using partial functions with incomparable image sets. The monoid considered by Thompson was \(\operatorname{tot} M_{2,1}\).

Birget extended the definition in \cite{birget2020monoid} to higher-dimensional analogs, \(nM_{k,1}\), which extend \(nG_{k,1}\). It is evident that the aforementioned submonoids extend in the same manner. We focus on the multi-rooted generalizations \(\operatorname{tot} nM_{k,r}\) of the non-partial monoids.

The variety of J\'{o}nsson-Tarski algebras was first introduced in \cite{jonsson1961two} as an example of a variety whose free objects exhibit unusual properties related to the generation of these free objects. Some work has been done on generalizations of the concept, such as in \cite{dudek1979universal, goetz1960bases}, which define varieties \(\mc{A}_{m,n}\). For our purposes, we will use \(\mc{A}_{1,n}\) and its higher-dimensional analogs, which are not equivalent to \(\mc{A}_{m,n}\) for dimension \(m\).

Higman worked with J\'{o}nsson-Tarski algebras and the generalizations \(G_{n,r}\) in \cite{higman1974finitely}, and this was extended by Mart\'{i}nez-P\'{e}rez and Nucinkis in \cite{martinez2013bredon}.  We show that analogously, \(\operatorname{tot} nM_{k,r}\) is the endomorphism monoid of an \(r\)-generated free algebra.

\begin{theorem}\label{main1}
For \(r, n, k \in \mathbb{N}\), with \(r, n \geq 1\) and \(k \geq 2\), the following monoids are isomorphic:
\begin{enumerate}
\item[i)] \(\operatorname{tot} nM_{k,r}\).
\item[ii)] The endomorphism monoid of an \(r\)-generated free algebra in the \(n\)-dimensional \(k\)-ary J\'{o}nsson-Tarski variety of \cref{def:nVk}.
\end{enumerate}
\end{theorem}

The isomorphism established in Theorem \ref{main1} is used to prove our second main result.

\begin{theorem}\label{main2}
For \(r, n, k \in \mathbb{N}\), with \(r, n \geq 1\) and \(k \geq 2\), the monoid \(\operatorname{tot} nM_{k,r}\) is finitely presented.
\end{theorem}

We give this presentation explicitly in the case of tot\(M_{2,1}\) in \cref{explictpresfinal}.
\begin{center}
\textbf{Acknowledgements}

We would like to thank Matt Brin and Collin Bleak for their input and advice in the initial construction of this paper. In particular, the proof of the finite presentability results in this paper is based on Thompson's original proof in the one-dimensional case, which we were made aware of by Matt Brin. 
The publicly available Python package (hosted on GitHub) for working in \(V\), authored by Nathan Barker, Andrew Duncan and David Robertson was of great assistance in finding words for the deferments in \cref{explicit}.
\end{center}

\section{Definitions of Objects}

In this section, we introduce the definitions of the relevant monoids and the algebraic objects involved. We keep these definitions brief; for a more detailed introduction to $\operatorname{tot}M_{k,1}$, see \cite{birget2009monoid}.

\subsection{Notes on Notation}

For the purposes of this paper, we establish the following notational conventions:

\begin{itemize}
    \item \textbf{Sets and Indices:} We identify the set of natural numbers $\mathbb{N}$ with the first infinite ordinal $\omega = \{0, 1, 2, \ldots\}$. 
    \item \textbf{Function Composition:} We compose functions from left to right, and write functions on the right of their inputs.
    \item \textbf{Restriction and Conjugation:}
    \begin{itemize}
        \item For a function $f$, we denote by $f|_A$ the restriction of $f$ to the set $A$.
        \item For elements $f, g$ of a semigroup with $g$ invertible, we define the conjugation of $f$ by $g$ as $f^g := g^{-1} \circ f \circ g$.
    \end{itemize}
    \item \textbf{Words and Free Monoids:} For an alphabet $A$, we denote the free monoid on $A$ by $A^*$. We call the elements of $A$ ``letters" and the elements of $A^*$ ``words". The set $A^{\omega}$ is the set of infinite words. The length of a word $w$ is written as $|w|$.
    \item \textbf{Indexing and Tuples:} We view words of length $m$ as tuples in $A^m$ and as functions from $\{0, \dots, m-1\}$ to $A$. We write $(i)w$ for the $i$-th letter of a word $w$, indexed from $i=0$. This notation is also used for tuples. The empty word is denoted by \(\varepsilon\).
    \item \textbf{Roots and Generators:} We write numbers in mathbf ($\mathbf{0}, \mathbf{1}, \ldots$) when they are intended to represent roots or generator indices. This is purely for the purposes of readability and there is no formal difference between \(\mathbf{0}\) and \(0\).
\end{itemize}

\subsection{Words, Trees, and Prefix Codes}

We begin by establishing the necessary terminology for words and structures, then proceed to define the monoids and the relevant variety of algebras from Theorem 1.1.

\begin{definition}[Alphabets and $nT_{k,r}$]
Suppose that $k \geq 2$ and $r, n \geq 1$. Let $X_r := \{\mathbf{0}, \mathbf{1}, \ldots, \mathbf{r-1}\}$, $A_k := \{0, 1, \ldots, k-1\}$, and $nT_{k,r} := X_r \times (A_k^*)^n$.
\end{definition}

The monoid $(A_k^*)^n$ acts on the set $nT_{k,r}$ by right multiplication on its second entry.
The set $nT_{k,r}$ is the $r$-rooted, $n$-dimensional $k$-ary analogue of the binary tree's vertices. We think of there being an edge from an element of $nT_{k,r}$ to another if the latter can be obtained from the former by appending a single letter to one of the words in its second component.
Note that this set is generally not a tree, and elements often have multiple ``parent" nodes (for instance, in $2T_{2,1}$, the element $(\mathbf{0},(0,0))$ can be obtained from $(\mathbf{0},(\varepsilon,0))$ by appending $0$ to the first component, or from $(\mathbf{0},(0,\varepsilon))$ by appending $0$ to the second component).

\begin{definition}[Cantor Spaces]
If $k \geq 2$ and $n \geq 1$, then we define $\mathfrak{C}_{n, k}:=(A_k^\omega)^n$ to be the $k$-ary $n$-dimensional Cantor space.
Moreover, if $r \geq 1$, then we define $\mathfrak{C}_{n, k, r}:=X_r \times \mathfrak{C}_{n, k}$ to be the $r$-rooted $k$-ary $n$-dimensional Cantor space, equipped with the usual product topology.
\end{definition}

These spaces are famously compact metrizable spaces and have a basis consisting of the clopen sets defined in \cref{def:cones}.

\begin{definition}[Shrubs, Shrubberies, and Depth]
Suppose that $k \geq 2$ and $n,r \geq 1$.
We call an element of $(A_k^* \cup A_k^\omega)^n$ a \textit{shrub}, and an element of $X_r \times (A_k^* \cup A_k^\omega)^n$ a \textit{shrubbery}.
In particular, the sets $nT_{k,r}$ and $\mathfrak{C}_{n,k,r}$ consist of shrubberies.
A shrubbery $w$ is a pair, the first entry being the \textit{root} $w^{\rt}$ and the second entry being the \textit{shrub} $w^{\shrub}$.

We define the \textit{depth} $|s|_{\text{depth}}$ of a shrub $s \in (A_k^* \cup A_k^\omega)^n$ to be the least $D \in \mathbb{N} \cup \{\omega\}$ such that for all $i < n$, we have $|(i)s| \leq D$.
\end{definition}

From the action of $A_k^*$ on $nT_{k,r}$, we can act on a finite shrubbery with a finite shrub. We extend this action to infinite shrubs in the natural fashion (we will not act on infinite shrubberies).

\begin{definition}[The Partial Order]\label{shruborder}
If $u, v$ are words, shrubs, or shrubberies, then we define a partial order $u \leq v$ if there is a shrub $s$ such that $u s = v$. In particular, for words, this is the usual ``prefix" partial order (as one dimensional shrubs are just words).
\end{definition}

\begin{definition}[Cones]\label{def:cones}
Suppose that $k \geq 2$ and $n,r \geq 1$. For a finite shrubbery $w$, we define the \emph{cone} under $w$ as
\[ w \mathfrak{C}_{n, k} := \{ u \in \mathfrak{C}_{n, k, r} \mid u \geq w \}.\]
This is the set of all shrubberies in $\mathfrak{C}_{n,k,r}$ that share the same root as $w$ and such that the words in their shrub component have the corresponding words in $w^{\shrub}$ as a prefix.
\end{definition}

One very natural notion in $nT_{k,r}$ is that of a complete prefix code. In the 1-dimensional case, this is equivalent to a finite complete antichain in the order on the words. In the finite case this definition agrees with the ``maximal joinless codes" of Birget in \cite{birget2009monoid,birget2020monoid}. We adopt the terminology of ``complete prefix code", as is common in works on the $1$-dimensional Higman-Thompson groups.

\begin{definition}[Pseudotrees and Complete Prefix Codes]\label{depthdef}
Suppose that $k \geq 2$ and $n,r \geq 1$.
For a finite subset $S$ of $nT_{k, r}$, we define $\operatorname{Leaves}(S)$ to be the set of maximal elements of $S$ with respect to the order in \cref{shruborder}.
A \emph{pseudotree} $T$ of $nT_{k,r}$ is a finite downwards closed subset (using the order of \cref{shruborder}) such that the cones of the leaves of $T$ partition $\mathfrak{C}_{n,k,r}$.
In this case, the set $\operatorname{Leaves}(T)$ is called a \emph{complete prefix code}.
We define the \textit{depth} of a pseudotree $T$ by $|T|_{\text{depth}} := \max \{|w^{\shrub}|_{\text{depth}} \mid w \in T \}$.
\end{definition}

In other words, the depth of $T$ is the length of the longest word in the shrubs of the shrubberies of $T$. Note again that a pseudotree is not necessarily a tree (it is a tree in the case $n=1$).

\begin{remark}
Complete prefix codes are finite because $\mathfrak{C}_{n,k,r}$ is a compact space, and any open partition of a compact space must be finite.
\end{remark}

Complete prefix codes are in bijective correspondence with pseudotrees. In particular, the leaves of any pseudotree are a complete prefix code, and every complete prefix code $P$ is the set of leaves of a unique pseudotree.

\begin{definition}[Pseudotree of a Prefix Code]
Suppose that $k \geq 2$ and $n,r \geq 1$.
If $P \subseteq nT_{k,r}$ is a complete prefix code, then we define $\operatorname{pst}(P)$ to be the unique pseudotree whose maximal elements are $P$. This pseudotree is precisely the downwards closure of $P$ in the order from \cref{shruborder}.
\end{definition}

The standard process for generating and modifying trees, often called expansion, is very useful when defining and working with higher dimensional pseudotrees. This idea is commonly used in work with Higman-Thompson groups and their extensions, for instance in \cite[Lemma 3.3]{lawson2021polycyclic}.
Later we will also use expansions to modify free generating sets of J\'{o}nsson-Tarski algebras, similar to Chapter 2 of \cite{higman1974finitely}.
We formulate expansions using multisets because, for technical reasons, the results in \cref{expanding_prefix_codes} and \cref{lem:FreeGenExpansion} will fail if we only consider usual sets.
This failure could also be avoided by requiring all sets involved to consist only of incomparable elements or the equivalent in the case of \cref{lem:FreeGenExpansion}.

\begin{definition}[Expansions]\label{expansions}
Suppose that $k \geq 2$ and $n,r \geq 1$.
Let $M$ be a multisubset of $nT_{k,r}$, and let $a \in M$ for some $i < n$.
For each $0 \leq l < k$, define $a_l \in nT_{k,r}$ such that
\begin{enumerate}
    \item $a_l^{\rt}=a^{\rt}$,
    \item $(i)a_l^{\shrub}=(i)a^{\shrub} \cdot l$,
    \item for all $i' \neq i$, we have $(i')a_l^{\shrub}=(i')a^{\shrub}$.
\end{enumerate}
The \emph{elementary expansion} of $M$ about $a$ in dimension $i$ is the multiset obtained from $M$ by removing one copy of $a$ from $M$ and adding a copy of each $a_l$ for $l < k$. An expansion is then the result of a sequence of elementary expansions.
\end{definition}

\begin{proposition}\label{expanding_prefix_codes}
    Suppose that $k \geq 2$, $n,r \geq 1$, and $M$ is a multisubset of $nT_{k,r}$.
    Any expansion of $M$ is a complete prefix code if and only if $M$ is a complete prefix code.
\end{proposition}
\begin{proof}
    An elementary expansion of a shrubbery $a$ replaces the cone $a\mathfrak{C}_{n,k}$ with the union of $k$ disjoint subcones $\bigcup_{l=0}^{k-1} a_l\mathfrak{C}_{n,k}$. Since this union is equal to the original cone, the replacement preserves the property that the cones of the multiset elements partition (or don't partition) the entire space $\mathfrak{C}_{n,k,r}$. The result follows by induction on the sequence of elementary expansions.
\end{proof}

For example, when $n=2$, $k=2$, $r=1$, and $M$ is the multiset \[\{(\mathbf{0}, (01,11)),(\mathbf{0},(01,11)),(\mathbf{0},(\varepsilon, 111))\}_{\text{multiset}},\] the elementary expansion of $M$ in dimension $1$ about the first copy of $(\mathbf{0}, (01,11))$ is the multiset
\[ \{(\mathbf{0}, (01,11)),(\mathbf{0},(01,110)),(\mathbf{0},(01,111)),(\mathbf{0},(\varepsilon, 111)) \}_{\text{multiset}}. \]
Note that dimension $1$ is the second of the two dimensions here due to zero indexing.

We denote the monoid of continuous maps from $\mathfrak{C}_{n, k,r}$ to itself by $\operatorname{End}(\mathfrak{C}_{n, k,r})$. The main monoid of interest is a submonoid of this monoid.

\begin{definition}[Prefix Exchange Maps and $\operatorname{tot} nM_{k,r}$]\label{def:prefix_exchange}
Suppose that $k \geq 2$ and $n,r \geq 1$, and that we have a pair $(D,h)$ where $D$ is a pseudotree of $nT_{k,r}$ and $h: \operatorname{Leaves}(D) \to nT_{k,r}$. We call such a pair a \emph{pseudotree pair}. We define the map $f:\mathfrak{C}_{n, k,r} \to \mathfrak{C}_{n, k,r}$ represented by this pseudotree pair by
$$ (d \cdot w) f = ((d)h) \cdot w \quad \text{for all } d \in \operatorname{Leaves}(D) \text{ and shrubs } w \in \mathfrak{C}_{n, k}. $$
We define $\operatorname{tot} nM_{k,r}$ to be the monoid of maps $f \in \operatorname{End}(\mathfrak{C}_{n, k,r})$ represented by pseudotree pairs.
\end{definition}

This monoid is equivalent to $\operatorname{tot} M_{k,1}$ as defined by Birget in \cite{birget2009monoid} for $n=1$ and $r=1$. It is also equivalent to the submonoid of $nM_{k,r}$ defined in \cite{birget2020monoid} that contains only the non-partial functions. The groups $nG_{k,r}=nV_{k,r}$ are then defined to be the groups of units of \(\operatorname{tot} n M_{k,r}\).
Comparing our definition to a common definition of $nG_{k,r}$, where we map the leaves of a pseudotree via a bijection to the leaves of another pseudotree, here we can map the leaves of a pseudotree to any set of shrubberies. 

\subsection{Varieties}

In \cite{jonsson1961two}, J\'{o}nnson and Tarski define a variety \(K_2\), which was constructed specifically to not have the properties: 
\begin{enumerate}
    \item An algebra in a variety \(K\) which is freely generated by a finite set cannot be generated by a set with fewer elements;
    \item If an algebra in a variety \(K\) is freely generated by a finite set, then any generating set of the same size is a free generating set.
\end{enumerate}

We now define the \(n\)-dimensional \(k\)-ary J\'{o}nnson-Tarski varieties that we will use for \cref{main1}. These varieties have appeared before in \cite{martinez2013bredon}. 
The idea of a 1-dimensional J\'{o}nnson-Tarski variety is to represent bijections from \(A^k\) to \(A\)  for some set \(A\), so we have a \(k\)-ary operation (the bijection) and \(k\) unary operations which `undo' the \(k\)-ary operation. It is not difficult to see that the general \(n\mc{J}_{k}\) also do not have properties (1) and (2) and this will come up in later proofs.

\begin{definition}[The Varieties] \label{def:nVk}
Let \(n\mc{J}_{k}\) be the variety of algebras with \(n\) \(k\)-ary operations \((\l_i)_{i<n}\) and \(nk\) unary operations \((\a_{i,j})_{i<n, j<k}\), satisfying the following equations:
\begin{enumerate}
    \item[i)] \((x\a_{i,0},\dots,x\a_{i,k-1})\l_i=x\) for all \(x\) and \(i<n\),
    \item [ii)] \((x_0,\dots,x_{k-1})\l_i\a_{i,j}=x_j\) for all \(x_0,\dots,x_{k-1},j<k\),
    \item [iii)] \(\alpha_{i, j}\alpha_{i', j'} =\alpha_{i', j'}\alpha_{i, j}\) for all \(i,i'<n\) with \(i\neq i'\) and \(j,j'<k\).
\end{enumerate}
We denote by  \(\mb{F}_{n,k,r}\), the algebra in the variety \(n\mc{J}_{k}\) freely generated by the set \(\{\mathbf{0},\mathbf{1},\ldots, \mathbf{r-1}\}\).
\end{definition}

The second of the monoids from \cref{main1} is the endomorphism monoid \(\End(\mb{F}_{n,k,r})\) of the algebra \(\mb{F}_{n,k,r}\). In the next section we solve the word problem in these algebras so we can more easily work with endomorphisms.

\section{Free objects in the multidimensional J\'{o}nsson-Tarski varieties}

We need to be able to describe the elements of $\mathbb{F}_{n,k,r}$. Each element is obtained from the free generating set by a sequence of operations, so we start by establishing some important laws which are consequences of the definition of the variety $n\mathcal{J}_k$. 

\begin{proposition}\label{prop:nVk1}
Suppose that $n,r \geq 1$ and $k \geq 2$. The following laws hold for all algebras in the variety $n\mathcal{J}_k$.
For all $i, i' < n$ with $i \neq i'$ and $l, m < k$:
\begin{align}
\label{eq1} (x_0, \dots, x_{k-1})\lambda_i \alpha_{i', l} \alpha_{i, m} &= x_m \alpha_{i', l} \\
\label{eq2} (x_0 \alpha_{i', l}, \dots, x_{k-1} \alpha_{i', l}) \lambda_i &= (x_0, \dots, x_{k-1}) \lambda_i \alpha_{i', l} \\
\label{eq3} \Big((x_{0,0}, \dots, x_{0, k-1}) \lambda_i, &\dots, (x_{k-1,0}, \dots, x_{k-1, k-1}) \lambda_i\Big) \lambda_{i'}  \\
&= \Big((x_{0,0}, \dots, x_{k-1,0}) \lambda_{i'}, \dots, (x_{0,k-1}, \dots, x_{k-1,k-1}) \lambda_{i'}\Big) \lambda_i \nonumber.
\end{align}
\begin{figure}
\centering
\includegraphics[]{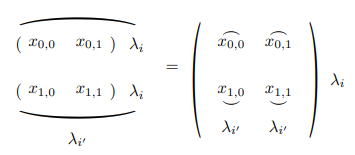}
 \caption{Pictorial representation of Proposition \ref{prop:nVk1} iii)}
 \label{fig:prop:nVK1iii}
\end{figure}

\end{proposition}
\begin{proof}
Proof of \cref{eq1}: By the defining laws of $n\mathcal{J}_k$ (specifically parts iii) and then ii) of \cref{def:nVk}), we have
\[(x_0,\dots,x_{k-1})\lambda_i \alpha_{i', l} \alpha_{i,m} = ((x_0,\dots,x_{k-1})\lambda_i \alpha_{i,m}) \alpha_{i',l} = x_{m} \alpha_{i', l}.\]

Proof of \cref{eq2}:  Applying \cref{eq1} followed by the first law of \cref{def:nVk} yields:
\begin{align*}
(x_0 \alpha_{i', l},\dots,x_{k-1} \alpha_{i', l})\lambda_i &= ((x_0,\dots,x_{k-1})\lambda_i \alpha_{i', l} \alpha_{i,0},\dots,(x_0,\dots,x_{k-1})\lambda_i \alpha_{i', l} \alpha_{i,k-1})\lambda_i\\
 &= (x_0,\dots,x_{k-1})\lambda_i \alpha_{i', l}.
\end{align*}

Proof of \cref{eq3}: Let
\[L := \Big((x_{0,0},\dots,x_{0,k-1})\lambda_i,\dots,(x_{k-1,0},\dots,x_{k-1,k-1})\lambda_i\Big)\lambda_{i'},\]
\[R := \Big((x_{0,0},\dots,x_{k-1,0})\lambda_{i'},\dots,(x_{0,k-1},\dots,x_{k-1,k-1})\lambda_{i'}\Big)\lambda_i.\]
We show that $L=R$.
First, apply the $\alpha$ operations to $L$. By applying Definition \ref{def:nVk} iii) and ii), we have:
\[ L \alpha_{i,l} \alpha_{i',m} = L \alpha_{i',m} \alpha_{i,l} = x_{m,l} = R \alpha_{i,l} \alpha_{i',m} \quad \text{for all } l, m < k. \]
We can show that \(L=R\) by iterating the first law of $\ref{def:nVk}$. For all \(l<k\) we have:
\begin{align*}
    L \alpha_{i,l} &= (L \alpha_{i,l} \alpha_{i',0}, L \alpha_{i,l} \alpha_{i',1},\ldots, L \alpha_{i,l} \alpha_{i',k-1})\lambda_{i'} \\
    &= (R \alpha_{i,l} \alpha_{i',0}, R \alpha_{i,l} \alpha_{i',1},\ldots, R \alpha_{i,l} \alpha_{i',k-1})\lambda_{i'} \\
    &= R \alpha_{i,l}.
\end{align*}
Thus:
\[ L = (L \alpha_{i,0}, L \alpha_{i,1},\ldots, L \alpha_{i,k-1})\lambda_{i} = (R \alpha_{i,0}, R \alpha_{i,1},\ldots, R \alpha_{i,k-1})\lambda_{i} = R.\]
\end{proof}

We now describe an example of an object in the variety $n\mathcal{J}_{k}$. We will use this object to describe our free algebras and to show that certain elements of $\mathbb{F}_{n,k,r}$ are distinct.

\begin{example}\label{example_source}
    Suppose that $n \geq 1$ and $k \geq 2$.
    Let $\mathcal{P}(\mathfrak{C}_{k}^n)$ be the set of subsets of $\mathfrak{C}_{k}^n$.
    We give $\mathcal{P}(\mathfrak{C}_{k}^n)$ the structure of an $n\mathcal{J}_{k}$ algebra as follows:

    \begin{enumerate}
        \item For $i < n$ and $S_0,S_1,\ldots,S_{k-1} \subseteq \mathfrak{C}_{k}^n$, we define
        \[ (S_0,S_1,\ldots,S_{k-1})\lambda_{i}:=\makeset{(x_0,\ldots,x_{n-1})\in \mathfrak{C}_{k}^n}{there is \(l<k\) and \(x_i'\) such that \(x_i=lx_i'\) and \\
    \((x_0,\ldots,x_{i-1},x_i',x_{i+1},x_{n-1})\in S_l\)}.\]
        \item For $i < n$, $j < k$ and $S \subseteq \mathfrak{C}_{k}^n$, we define
        \[ (S)\alpha_{i,j}:=\makeset{(x_0,\ldots,x_{n-1})\in \mathfrak{C}_{k}^n}{\((x_0,\ldots,x_{i-1},jx_i,x_{i+1},x_{n-1})\in S\)}.\]
    \end{enumerate}
    To see that this algebra belongs to $n\mathcal{J}_{k}$, we check that the required laws hold.
    For $S,S_0,S_1,\ldots,S_{k-1} \subseteq \mathfrak{C}_{k}^n$, $i,i'<n$ and $j,j'<k$ with $i\neq i'$, we observe:
    \begin{align*}
        (S\alpha_{i,0},\dots,S\alpha_{i,k-1})\lambda_i&=\makeset{(x_0,\ldots,x_{n-1})\in \mathfrak{C}_{k}^n}{there is \(l<k\) and \(x_i'\) such that \(x_i=lx_i'\) and \\
    \((x_0,\ldots,x_{i-1},x_i',x_{i+1},x_{n-1})\in S\alpha_{i,l}\)}\\
    &=\makeset{(x_0,\ldots,x_{n-1})\in \mathfrak{C}_{k}^n}{there is \(l<k\) and \(x_i'\) such that \(x_i=lx_i'\) and \\
    \((x_0,\ldots,x_{i-1},lx_i',x_{i+1},x_{n-1})\in S\)}=S
    \end{align*}
    \begin{align*}
        (S_0,\dots,S_{k-1})\lambda_i\alpha_{i,j}&=\makeset{(x_0,\ldots,x_{n-1})\in \mathfrak{C}_{k}^n}{\((x_0,\ldots,x_{i-1},jx_ix_{i+1},x_{n-1})\in (S_0,\dots,S_{k-1})\lambda_i\)}\\
  &=\makeset{(x_0,\ldots,x_{n-1})\in \mathfrak{C}_{k}^n}{there is \(l<k\)  and \(x_i'\) such that \(jx_i=lx_i'\)\\ and \((x_0,\ldots,x_{i-1},x_i',x_{i+1},x_{n-1})\in S_l\)}=S_j
    \end{align*}
    \begin{align*}
        (S)\alpha_{i, j}\alpha_{i', j'} &=\makeset{\mathbf{x}\in \mathfrak{C}_{k}^n}{the tuple obtained from \(\mathbf{x}\) by prepending a \(j\) in position \(i\)\\ and a \(j'\) in position \(i'\) belongs to \(S\)}=S\alpha_{i', j'}\alpha_{i, j}.
    \end{align*}
\end{example}

The algebras we use in the main theorem are the free algebras in these varieties. Free algebras can be constructed in any variety, but we give an explicit construction so we can analyze them.

\begin{definition}[Free Algebras]\label{sorting_out_free_algebra}
For each $k \geq 2$, $n,r \geq 1$, let $\mathbb{F}_{n,k,r}$ be the free algebra (in the variety $n\mathcal{J}_{k}$) with free generating set $\{\mathbf{0},\mathbf{1} ,\ldots, \mathbf{r-1}\}$.
That is, if $nA_{k, r}$ is the smallest set of formal strings such that
\begin{enumerate}
    \item each of $\mathbf{0},\mathbf{1} ,\ldots, \mathbf{r-1}$ is an element of $nA_{k, r}$,
    \item if $x_0, x_1, \ldots, x_{k-1} \in nA_{k, r}$ and $d < n$, then the string $\texttt{(}x_0, x_1, \ldots, x_{k-1}\texttt{)}\lambda_d$ belongs to $nA_{k, r}$,
    \item if $x \in nA_{k, r}$, $d < n$ and $j < k$, then the string $\texttt{(}x\texttt{)}\alpha_{d, j}$ belongs to $nA_{k, r}$,
\end{enumerate}
where $\alpha_{d, j}$, $\lambda_d$, commas, and parenthesis are treated as formal symbols for $d < n$ and $j < k$.
Then $\mathbb{F}_{n,k,r}$ is the quotient of $nA_{k,r}$ by the least congruence $=_{\mathbb{F}_{n,k,r}}$ containing the relations defining $n\mathcal{J}_{k}$.
\end{definition}

We now give a pseudo normal form as well as a means for determining whether or not two elements of $nA_{k,r}$ are equal in $\mathbb{F}_{n,k, r}$.

\begin{proposition}\label{free_understanding}
Suppose that $k \geq 2$,  $n,r \geq 1$,  $a,b \in nA_{k,r}$ and $m$ is larger than the number of $\lambda$ symbols in $a$ and $b$. 
Let $C_m$ be the set of all words over the alphabet $\makeset{\alpha_{i,j}}{\(i < n, j < k\)}$ which contain $m$ letters of the form $\alpha_{i,j}$ for all $i < n$.
The following hold:
\begin{enumerate}
    \item For every element of $ \mathbb{F}_{n,k, r}$, there is an element of $nA_{k, r}$ representing it which can be made using a sequence of $\alpha$-type operations applied to various generators, followed by a sequence of $\lambda$-type operations;
    \item For all $\alpha \in  C_m$, the elements $(a)\alpha$ and $(b)\alpha$ of $nA_{k,r}$ are equal in $\mathbb{F}_{n,k,r}$ to elements with no $\lambda$ symbols;
    \item Moreover $a =_{\mathbb{F}_{n,k,r}} b$ if and only if
$(a)\alpha =_{\mathbb{F}_{n,k,r}} (b)\alpha$ for all $\alpha \in  C_m$;
\item If $a$ and $b$ contain no $\lambda$ symbols, then they are each a generator followed by a string in the letters $\makeset{\alpha_{i,j}}{\(i < n, j < k\)}$. In this case, we have $a =_{\mathbb{F}_{n,k,r}} b$ if and only if the corresponding generators are the same and the strings in $\makeset{\alpha_{i,j}}{\(i < n, j < k\)}$ represent the same element of the monoid
\[ \prod_{i < n} \makeset{\alpha_{i,j}}{\(j < k\)}^*. \]
\end{enumerate}

\end{proposition}
\begin{proof}
The first claim is immediate from Definition \ref{def:nVk} ii) and Proposition \ref{prop:nVk1} (2).

We next need to show that for all $\alpha \in  C_m$, the elements $(a)\alpha$ and $(b)\alpha$ of $nA_{k,r}$ are equal in $\mathbb{F}_{n,k,r}$ to elements with no $\lambda$ symbols. This also follows from Definition \ref{def:nVk} ii) and Proposition \ref{prop:nVk1} (2).

We next show that $a =_{\mathbb{F}_{n,k,r}} b$ if and only if
$(a)\alpha =_{\mathbb{F}_{n,k,r}} (b)\alpha$ for all $\alpha \in  C_m$. The forward implication is clear.
The reverse implication follows by repeatedly applying the law from \cref{def:nVk} i) to generate $a$ and $b$ using the elements $(a)\alpha = (b)\alpha$ for $\alpha \in  C_m$.

Definition \ref{def:nVk} iii) (which asserts that $\alpha$ operations commute across dimensions.
It follows that the operations \(\makeset{\alpha_{i,j}}{$i < n, j < k$}\) satisfy the relations of the monoid $ \prod_{i < n} \makeset{\alpha_{i,j}}{$j < k$}^*$.
Thus we need only show that there is an $n\mathcal{J}_{k}$ algebra and elements $S_0,\ldots S_{r-1}$ such that for all $\alpha, \beta \in \makeset{\alpha_{i,j}}{$i < n, j < k$}^*$ and $l,m < r$
we have
\[a_l \alpha \neq a_{m} \beta\]
whenever $a_l \neq a_m$ or $\alpha$ and $\beta$ represent distinct elements of the monoid $ \prod_{i < n} \makeset{\alpha_{i,j}}{$j < k$}^*$.
For each $l < r$, we define
\[S_l:=\makeset{(x_0,x_1,\ldots, x_{n-1})\in \mathfrak{C}_k^n}{for all \(i<n\), \(x_i=v_i10^l1v_i10^{2l}1v_i10^{3l}1\ldots \) for some finite word \(v_i\)}.\]
These elements of the algebra from \cref{example_source} have the required property, as both the number $l$ and each of the components of $\alpha$ are recoverable from the set $S_l \alpha$.
\end{proof}

With the preceding proposition, we now have a (non-unique) form by which we can represent elements of $\mathbb{F}_{n,k,r}$, with `all the lambdas at the end'. Note that the lambdas aren't actually at the end of the string representation, but they are applied last (for example, the format $((\mathbf{0},\mathbf{1})\lambda,(\mathbf{2})\alpha)\lambda$ is allowed).

It is worth keeping in mind that the technique used to check equality of elements of the free $n\mathcal{J}_{k}$ algebras is very similar to the technique we will use to check equality of elements of our monoids when we show finite presentability in \cref{determinewithprjections}.
\section{Constructing the Isomorphism}

Now we can start work on developing the tools with which we can prove \cref{main1}. We start with a natural embedding of $nT_{k,r}$ into our free algebras. The idea is that we map the roots of the shrubberies to the free generators, and the shrubs attached to the roots correspond with applying unary operations in the algebra.

\begin{definition}[Tree Inside the Algebra]\label{tree_map_def}
Suppose that $k \geq 2$, and $n,r \geq 1$. Define a map $\tree_{n,k,r}:nT_{k,r}\to \mathbb{F}_{n,k,r}$ by
\[(\mathbf{m},(s_0, s_1, \ldots, s_{n-1}))\tree_{n,k,r}=(\mathbf{m})\prod_{i=0}^{n-1}\prod_{j=0}^{|s_i|-1}\alpha_{i,(j)s_i}\]
where the product is composition.
Equivalently, let $\phi_{\alpha}: (\{0, 1, \ldots, k-1\}^*)^n \to \makeset{f}{\(f:\mathbb{F}_{n,k,r} \to \mathbb{F}_{n,k,r}\) }$ be the unique semigroup homomorphism such that for all $i<n$ and $j<k$, $\phi_{\alpha}$ sends the shrub with a $j$ in the $i$th coordinate, and the empty word in all other coordinates, to the unary operation $\alpha_{i, j}$.
Then define the map $\tree_{n,k,r}:nT_{k,r}\to \mathbb{F}_{n,k,r}$ by
\[(\mathbf{m},s)\tree_{n,k,r}= (\mathbf{m})((s)\phi_{\alpha}).\]
Moreover, \cref{free_understanding} implies that $\tree_{n,k,r}$ is injective.
\end{definition}

\begin{lemma}\label{lem:FreeGenExpansion}
Suppose that $k \geq 2$, and $n,r \geq 1$.
Suppose further that $M$ is a finite multisubset of $nT_{k, r}$ and $E$ is an elementary expansion of $M$.
The multiset $(M)\tree_{n, k, r}$ is a free generating set for $\mathbb{F}_{n,k, r}$ if and only if the multiset $(E)\tree_{n, k, r}$ is a free generating set for $\mathbb{F}_{n,k, r}$.
(If a multiset contains an element at least twice then we consider the equality of these two copies to be a non-trivial relation, precluding it from being a free generating set.)
\end{lemma}
\begin{proof}
$(\Rightarrow)$ Suppose that $E$ is obtained from $M$ as an elementary expansion about $a \in M$.
Let $i<n$ be the dimension expanded, and let $\left\{ a_l \mid 0 \leq l < k \right\}$ be as in \cref{expansions}. The expansion replaces $(a)\tree_{n,k,r}$ with $k$ new generators $\left\{ (a_l)\tree_{n,k,r} \mid 0 \leq l < k \right\}$ where $(a_l)\tree_{n,k,r} = ((a)\tree_{n,k,r}) \alpha_{i,l}$.

Let $\Gamma$ be an algebra in the variety $n\mathcal{J}_{k}$ and $\phi: (E)\tree_{n, k, r} \to \Gamma$ be a function. We must show that $\phi$ extends uniquely to a homomorphism $\tilde{\phi}: \mathbb{F}_{n,k, r}\to \Gamma$.
We define a map $\phi':(M)\tree_{n, k, r} \to \Gamma$ which agrees with $\phi$ on their common domain $(M \setminus \{a\})\tree_{n,k,r}$ and is defined on $(a)\tree_{n,k,r}$ as:
\[((a)\tree_{n,k,r})\phi' = ((a_0)\phi, \ldots, (a_{k-1})\phi)\lambda_i.\]
Since $(M)\tree_{n, k, r}$ is a free generating set, $\phi'$ extends uniquely to a homomorphism $\tilde{\phi}: \mathbb{F}_{n,k, r}\to \Gamma$.
The second law of Definition~\ref{def:nVk} ensures that for any $l<k$, $((a)\tree_{n,k,r})\tilde{\phi} \alpha_{i,l} = (a_l)\phi$. Thus, $\tilde{\phi}$ is the unique extension of $\phi$.

$(\Leftarrow)$ This follows by a similar argument.
\end{proof}

Note if $M$ contains multiple copies of any of its elements, then it is not a complete prefix code. Similarly if $(M)\tree_{n, k, r}$ contains multiple copies of any of its elements, then it is not a free generating set.
We can use these observations to show that free generating sets in $\mathbb{F}_{n,k,r}$ correspond precisely to complete prefix codes, by expanding to a ``flat" shrubbery.

Certain prefix codes are easier to build and work with than others and this will be important later in the document (in particular in \cref{lem:gens}). As such we introduce these now and observe when they appear.

\begin{definition}[Root Expansions]\label{rootexpansions}
    Suppose that $k \geq 2$, and $n,r \geq 1$. We say that a complete prefix code is a \emph{root expansion} if it can be obtained from \(\{(\mathbf{0},(\varepsilon,\ldots,\varepsilon),(\mathbf{1},(\varepsilon,\ldots,\varepsilon),\ldots,(\mathbf{r-1},(\varepsilon,\ldots,\varepsilon)))\}\) by a sequence of elementary expansions. 
    We refer to a free generating set \(X\subseteq \mathbb{F}_{n,k,r}\) as a \emph{root expansion generating set} if \(X\subseteq \im(\operatorname{tree}_{n,k,r})\) and \((X)\operatorname{tree}_{n,k,r}^{-1}\) is a root expansion.
\end{definition}
In the case that \(n=1\), it is well known that all complete prefix codes are root expansions, however this fails in higher dimensions. Consider for example
\[\{(\mathbf{0},(\varepsilon, 0, 0)),(\mathbf{0},(0, 1, \varepsilon)),(\mathbf{0},(1, \varepsilon, 1)),(\mathbf{0},(0, 0, 1)), (\mathbf{0},(1, 1, 0))\}\]
when \(r=1\), \(k=2\), \(n=3\)
(taken from Remark 10.4 of \cite{elliott2021constructing}).

We now return to our free algebra.
\begin{lemma} \label{AntichainImisGenSet}
Suppose that $k \geq 2$, and $n,r \geq 1$. For any finite $A\subseteq nT_{k,r}$, the set $(A)\tree_{n,k,r}$ is a free generating set for $\mathbb{F}_{n,k,r}$ if and only if $A$ is a complete prefix code.
Moreover in this case $A$ has an expansion to a root expansion and has cardinality in \(r+(k-1)\mathbb{N}\).
\begin{proof}
By repeatedly applying elementary expansions to $A$ and using \cref{lem:FreeGenExpansion}, we can obtain a multiset $M$ of shrubberies such that:
\begin{enumerate}
    \item $M$ is a complete prefix code if and only if $A$ is (\cref{expanding_prefix_codes}),
    \item $(A)\tree_{n,k,r}$ is a free generating set for $\mathbb{F}_{n,k, r}$ if and only if $(M)\tree_{n,k,r}$ is (\cref{lem:FreeGenExpansion}),
    \item there is $N \in \mathbb{N} $ such that for all $s \in M$, $j<n$, $|(j)s^{\shrub}|=N$.
\end{enumerate}
Since all shrubs in $M$ are of the same uniform depth $N$, the set $M$ is a complete prefix code if and only if it is precisely the set of all $n$-shrubberies of depth $N$, each appearing exactly once in \(M\).

By the definitions of the free algebra and $\tree_{n,k,r}$, the multiset $(M)\tree_{n,k,r}$ is a free generating set for $\mathbb{F}_{n,k, r}$ if and only if $M$ contains exactly one element for every shrubbery such that all the entries of its shrub have length $N$ (otherwise it would not generate all such elements or would contain duplicates). 
Hence, $A$ is a complete prefix code if and only if $(A)\tree_{n,k,r}$ is a free generating set.
Note also that in the case that \(M\) is a complete prefix code, it is also a root expansion.

 There are \(r k^{Nn}\) shrubberies all of whose shrubs consist of words of length \(N\). Moreover this set is constructed from \(A\) by elementary expansions which do not change the size of \(A\) modulo \(k-1\).
 Thus \(|A|\in r + (k-1)\mathbb{Z}\).
 A complete prefix code needs at least \(r\) elements to cover \(\mathfrak{C}_{n,k,r}\) so the result follows.
 
\end{proof}
\end{lemma}

We are now ready to define our map and show it is a homomorphism. The map will act on an element of $\operatorname{tot} nM_{k,r}$, represented by a pair $(D,h)$. The image is an endomorphism, which is uniquely defined by how it acts on the free generating set that corresponds with the leaves of $D$.

\begin{definition}\label{phi_def}
Suppose that $k \geq 2$, and $n,r \geq 1$.
Define a map $\phi_{n,k,r}:\operatorname{tot} nM_{k,r}\to \operatorname{End}(\mathbb{F}_{n,k,r})$ as follows.
If $f\in \operatorname{tot} nM_{k,r}$ is represented by a pair $(D,h)$, then for each $d\in \operatorname{Leaves}(D)$ we have
\[((d)\tree_{n,k,r})(f)\phi_{n,k,r}=((d)h)\tree_{n,k,r}.\]
See \cref{phidefined} for proof that this is well-defined.
\end{definition}

\begin{proposition}\label{phidefined}
    Suppose that $k \geq 2$, and $n,r \geq 1$.
    The map $\phi_{n,k,r}:\operatorname{tot} nM_{k,r}\to \operatorname{End}(\mathbb{F}_{n,k,r})$ from \cref{phi_def} is well-defined.
\end{proposition}
\begin{proof}
As $\operatorname{Leaves}(D)$ is a complete prefix code, Lemma \ref{AntichainImisGenSet} implies that $\phi_{n,k,r}$ is well-defined in terms of the representative pair $(D,h)$. We must now show that the choice of $(D, h)$ does not change $\phi_{n,k,r}$.

Let $(D_0,h_0)$ and $(D_1,h_1)$ be two pairs representing the same element $f\in \operatorname{tot} nM_{k,r}$.
Recall \cref{depthdef} and note that $P:=X_r\times (A_k^{|D_0|_{\text{depth}}+ |D_1|_{\text{depth}}})^n$ is a complete prefix code.
In particular $D=\operatorname{pst}(P)$ is the largest pseudotree of depth at most $|D_0|_{\text{depth}}+ |D_1|_{\text{depth}}$ and $D_0\cup D_1 \subseteq D$.
Define $h:\operatorname{Leaves}(D)\to nT_{k,r}$ by:
\[(x)h= \text{ the shrubbery }w\text{ such that }(x\mathfrak{C}_{n, k})f = w\mathfrak{C}_{n, k}.\]
Let $\psi_0$ and $\psi$ be the endomorphisms given by the definition of $(f)\phi_{n,k,r}$ for the two representations $(D_0,h_0)$ and $(D,h)$ respectively.
Since $\operatorname{Leaves}(D)$ is a free generating set for $\mathbb{F}_{n,k,r}$, any endomorphism is determined by its action on these elements. Let $d\in \operatorname{Leaves}(D)$.
It follows that there is a shrub $w$ and $d_0\in \operatorname{Leaves}(D_0)$ such that $d=d_0w$. Since both pairs represent $f$, it follows that $(d)h=(d_0)h_0\cdot w$.
Thus, letting $W$ be the composition of $\alpha$-operations corresponding to $w$, we have:
\begin{align*}
((d)\tree_{n,k,r})\psi&=((d)h)\tree_{n,k,r}\\
    &=((d_0)h_0\cdot w)\tree_{n,k,r}\\
    &=(((d_0)h_0)\tree_{n,k,r})W\\
    &=(((d_0)\tree_{n,k,r})\psi_0)W\\
    &=(((d_0)\tree_{n,k,r})W)\psi_0 & \text{(by Proposition \ref{prop:nVk1}, Eq. \ref{eq2})}\\
    &=((d_0w)\tree_{n,k,r})\psi_0\\
    &=((d)\tree_{n,k,r})\psi_0.
\end{align*}
Thus $\psi_0=\psi$. It follows by symmetry that $\psi_0=\psi=\psi_1$ and hence $\phi_{n,k,r}$ is well defined.
\end{proof}

We now show that this is a homomorphism. We pick representations of our elements $f,g\in \operatorname{tot} nM_{k,r}$ which simplify the representation of $f\circ g$.

\begin{lemma} \label{phihom}
For all $k \geq 2$, and $n,r \geq 1$, the map $\phi_{n,k,r}$ is a homomorphism.
\begin{proof}
Let $f,g\in \operatorname{tot} nM_{k,r}$ be represented by pseudotree pairs $(D_0,h_0)$ and $(D_1,h_1)$ respectively.
By expanding $D_0$ and $D_1$ if needed, we can assume without loss of generality that for every $d\in \operatorname{Leaves}(D_0)$, the shrub $(d)h_0$ is a leaf of $D_1$ (i.e., $(d)h_0 \in \operatorname{Leaves}(D_1)$).
Then $f \circ g$ is represented by the pair $(D_0,h_0\circ h_1)$.

Let $d\in \operatorname{Leaves}(D_0)$ be arbitrary. We have:
\begin{align*}
    ((d)\tree_{n,k,r})(f\circ g)\phi_{n,k,r} &= ((d)h_0\circ h_1)\tree_{n,k,r} \\
    &= ( ( (d)h_0 )\tree_{n,k,r} ) (g)\phi_{n,k,r} & \text{(by Def. \ref{phi_def} applied to $(g)\phi_{n,k,r}$)} \\
    &= ( ( (d)\tree_{n,k,r} ) (f)\phi_{n,k,r} ) (g)\phi_{n,k,r} & \text{(by Def. \ref{phi_def} applied to $(f)\phi_{n,k,r}$)}
\end{align*}
As $d\in \operatorname{Leaves}(D_0)$ was arbitrary, and $(\operatorname{Leaves}(D_0))\tree_{n,k,r}$ is a free generating set by \cref{AntichainImisGenSet}, it follows that $(f)\phi_{n,k,r} \circ (g)\phi_{n,k,r}=(f\circ g)\phi_{n,k,r}$.
\end{proof}
\end{lemma}

To prove that $\phi_{n,k,r}$ is in fact an isomorphism, we need only show bijectivity. To do this we need a better understanding of how free generating sets can be mapped in \(\mathbb{F}_{n,k,r}\). The following lemma will also be used in later sections.

\begin{lemma}\label{nice_generating_sets}
    Suppose that $k \geq 2$, and $n,r \geq 1$. If $f\in \operatorname{End}(\mathbb{F}_{n,k,r})$, then there are root expansion generating sets $G$ and $G'$ for $\mathbb{F}_{n,k,r}$ such that $G\cup G'\subseteq \im(\tree_{n,k, r})$, and \((G)f\subseteq G'\).
\end{lemma}
\begin{proof}
    Consider the set \(\{(\mathbf{0})f,\ldots, (\mathbf{r-1})f\}\). Each element of this set can be written in the pseudo normal form described in \cref{free_understanding}.
    Let \(N\) be the total number of \(\lambda\) type symbols in all of these pseudo normal forms.
    Using this \(N\), we build the set \(C_N\) as defined in \cref{free_understanding}.
    In particular, for all \(\alpha\in C_N\) and \(\mathbf{i}<r\), we have \((\mathbf{i})f\alpha =(\mathbf{i})\alpha f\) belongs to \(\im(\operatorname{tree}_{n,k,r})\).

    Note also that the set \(G_N:=\makeset{(\mathbf{i})\alpha}{\(\mathbf{i}<r\) and \(\alpha\in C_N\)}\) is a free generating set by iterated application of \cref{lem:FreeGenExpansion}.
    We also have \((G_N)f\cup G_N \subseteq \im(\operatorname{tree}_{n,k,r})\).

    Let \(N'\) be greater than the number of \(\alpha\) symbols needed to generate all elements of \((G_N)f\) from the free generating set \(\{\mathbf{0},\mathbf{1},\ldots, \mathbf{r-1}\}\).  As before, we build the set \(C_{N'}\) and a free generating set \(G_{N'}\).

    Let \(B'\) be the set of cones under the shrubberies in \((G_{N'})\operatorname{tree}_{n,k, r}^{-1}\).
    Similarly, let \(B\) be the set of cones under the shrubberies in \(((G_N)f)\operatorname{tree}_{n,k,r}^{-1}\). 
    By \cref{AntichainImisGenSet}, \(B'\) is a partition of \(\mathfrak{C}_{n,k,r}\).
    By the choice of \(N'\), each element of \(B\) is a subset of the cone under a unique element of \(((G_N)f)\operatorname{tree}_{n,k,r}^{-1}\).
    In particular, \(B'\) has subsets partitioning each element of \(B\) (by the choice of \(N'\)).
    For each \(x\in G_N\), let \(S_x\) be a set of shrubs such that 
    \[\makeset{b\in B'}{\(b\subseteq ((x)f)\operatorname{tree}_{n,k,r}^{-1} \mathfrak{C}_{n,k}\)}=\makeset{((x)f)\operatorname{tree}_{n,k,r}^{-1}\cdot s \mathfrak{C}_{n,k}}{\(s\in S_x\)}.\]
    As \(B'\) contains a partition of \(((x)f)\operatorname{tree}_{n,k,r}^{-1} \mathfrak{C}_{n,k}\), it follows that
    \[\makeset{s \mathfrak{C}_{n,k}}{\(s\in S_x\)}\]
    is a partition of \(\mathfrak{C}_{n,k}\) for all \(x\in G_{N'}\). So \(\{(\mathbf{0},s)|s\in S_x\}\) is a complete prefix code for \(\mathfrak{C}_{n,k,1}\). 
    From \cref{AntichainImisGenSet}, it follows that \(\makeset{(\mathbf{0},s)\operatorname{tree}_{n,k,1}}{$s\in S_x$}\) is a free generating set for \(\mathbb{F}_{n,k,1}\) for each \(x\).
    Recall the map \(\phi_\alpha\) from \cref{tree_map_def}.
    
    As \(\makeset{(\mathbf{0},s)\operatorname{tree}_{n,k,1}}{$s\in S_x$}\) is a free generating set for \(\mathbb{F}_{n,k,1}\), it follows that the set obtained from \(G_N\) by replacing \(x\) with \(\{(x)(s)\phi_\alpha|s\in S_x\}\) is also a free generating set.
    Perform this replacement for all \(x\in G_{N}\) to obtain a new free generating set \(G''\).
    By construction, we now have \((G'')f\subseteq G_{N'}\).

    Also by construction, the generating set \(G_{N'}\) is a root expansion. By \cref{AntichainImisGenSet}, \((G'')\operatorname{tree}_{n,k,r}^{-1}\) has an expansion which is a root expansion. So by a sequence of replacing elements \(x\in G''\) with \(\{x\alpha_{i,0},\ldots, x\alpha_{i,k-1}\}\) for various \(i\), we can transform \(G''\) into a root expansion (we do not do this yet). Let \(N''\in\N\) be larger than the number of steps in this sequence. 
 
    Perform \(N''\) expansions to \((G'')\operatorname{tree}_{n,k,r}^{-1}\) in every dimension, so the new free generating set corresponding to \(G''\) is \(G:=\makeset{g\alpha}{$g\in G'', \alpha\in C_{N''}$}\). By the choice of \(N''\), \((G)\operatorname{tree}_{n,k,r}^{-1}\) is an expansion of a root expansion and is hence a root expansion.
    Moreover, as \(f\) is a homomorphism, we have 
    \[\makeset{g\alpha}{$ g\in G'', \alpha\in C_{N''}$}f=\makeset{(g)f\alpha}{$ g\in G'',\alpha\in C_{N''}$}\subseteq \makeset{g\alpha}{$g\in G_{N'},\alpha\in C_{N''}$}.\]
    The result now follows using the root expansion generating sets \(G\) and \(G':=\makeset{g\alpha}{$g\in G_N', \alpha\in C_{N''}$}\).

\end{proof}

We can now show that \(\phi_{n,k,r}\) is a bijection.

\begin{proposition}\label{psidefined}
    Suppose that $k \geq 2$, and $n,r \geq 1$.
    The map $\phi_{n,k,r}$ is an isomorphism and so $ \operatorname{tot} nM_{k,r}\cong \operatorname{End}(\mathbb{F}_{n,k,r})$.
\end{proposition}
\begin{proof}
\cref{phihom} establishes that $\phi_{n,k,r}$ is a homomorphism. We next show that \(\phi_{n,k,r}\) is surjective.

Let \(f\in \operatorname{End}(\mathbb{F}_{n,k,r})\). By \cref{nice_generating_sets}, let \(G\) be a free generating set for \(f\) such that \(G\cup (G)f\subseteq \im(\operatorname{tree}_{n,k,r})\).
From \cref{AntichainImisGenSet}, it follows that \((G)\operatorname{tree}_{n,k,r}^{-1}\) is a complete prefix code.
Let \(D:=\pst((G)\operatorname{tree}_{n,k,r}^{-1})\) and let \(h:\operatorname{Leaves}(D)\to \mathbb{F}_{n,k,r}\) be defined by
\[((l)h)\operatorname{tree}_{n,k,r}=((l)\operatorname{tree}_{n,k,r})f.\]
It follows from the definition of \(\phi_{n,k,r}\) that \(f\) is the image under \(\phi_{n,k,r}\) of the element of \(\operatorname{tot} nM_{k,r}\) defined by \((D,h)\).

It remains only to show that \(\phi_{n,k,r}\) is injective. Suppose now that \(f,g\in \operatorname{tot} nM_{k,r}\) and \((f)\phi_{n,k,r}=(g)\phi_{n,k,r}\). We show that \(f=g\).
Let \(D\) be a large enough pseudotree that there are \(\tilde{f},\tilde{g}: \operatorname{Leaves}(D)\to nT_{k,r}\) with \((D,\tilde{f})\) and \((D,\tilde{g})\) defining \(f\) and \(g\) respectively. From the definition of \(\phi_{n,k,r}\) it follows that
for all \(d\in \operatorname{Leaves}(D)\), we have
\[((d)\tree_{n,k,r})(f)\phi_{n,k,r}=((d)\tilde{f})\tree_{n,k,r}\]
and 
\[((d)\tree_{n,k,r})(g)\phi_{n,k,r}=((d)\tilde{g})\tree_{n,k,r}.\]
Let \(h:=(f)\phi_{n,k,r} =(g)\phi_{n,k,r}\). For all \(d\in \operatorname{Leaves}(D)\), we now have
\[(d)\tilde{g}=((d)\tree_{n,k,r})h\tree_{n,k,r}^{-1}=(d)\tilde{f}.\]
Thus \(\tilde{g}=\tilde{f}\). It follows that \(f=g\) as required.

\end{proof}

\section{The Rewrite Perspective}

We now show that the monoids $\operatorname{tot} nM_{k,r}$ are finitely presented. We will use a method similar to Thompson's original proof of finite presentability for his monoid $\operatorname{tot} 1M_{2,1}$, which involves converting to and working with an anti-isomorphic monoid, which uses labelled free generating sets. An element of Thompson's group $V$ (the group of units of $\operatorname{tot} 1M_{2,1}$) can be thought of as a rewriting rule between bracketed expressions. For example, the prefix exchange map $0\mapsto 0$, $10\mapsto 11$, $11\mapsto 10$ is the rule $a(bc)\to a(cb)$.
These labelled free generating sets will be the analogue of bracketed expressions in this framework; they can also be thought of as matrices as we see later.

\begin{definition}[Labelled Free Generating Sets]
Suppose that $k\geq 2$, and $n,r\geq 1$.
Let $\mathcal{L}_{n,k, r}$ be the set of labeled root expansion generating sets for $\mathbb{F}_{n,k, r}$ with labels from a fixed infinite set. Formally, we view an element $L \in \mathcal{L}_{n,k, r}$ as a function $L: G \to \Lambda$, where $G$ is a free generating set for $\mathbb{F}_{n,k, r}$ and $\Lambda$ is the set of labels.
Let $\mathcal{L}_{n,k,r}^I\subseteq \mathcal{L}_{n,k,r}$ be the subset of labelled free generating sets such that different generators have different labels (i.e., the labeling function is injective).
\end{definition}

For example, when $n=1$, $k=2$, and $r=1$, the labelled generating set $\mathbf{0}\alpha_0\mapsto a$, $\mathbf{0}\alpha_1\alpha_0\mapsto b$, $\mathbf{0}\alpha_1\alpha_1\mapsto c$ represents the expression $a(bc)$.

For each $f\in \operatorname{End}(\mathbb{F}_{n,k, r}) \cong \operatorname{tot} nM_{k, r}$, let $(f)\sigma_{n, k, r}$ denote the binary relation on $\mathcal{L}_{n,k,r}$ defined by
\[(f)\sigma_{n, k, r}:= \makeset{(L_1, L_2)\in \mathcal{L}_{n,k,r}\times \mathcal{L}_{n,k,r}}{\(\text{for each } g \in \operatorname{dom}(L_2) \text{, we have } (g)f \in \operatorname{dom}(L_1)\) \\
\(\text{and } (g)f \text{ has the same label as } g \text{ under } L_1 \text{ and } L_2.\)}\]
Some readers may prefer to think of the condition defining $(f)\sigma_{n, k, r}$ as the equivalent claim that the map $f$ is an extension of a function from $\operatorname{dom}(L_2)$ to $\operatorname{dom}(L_1)$ that preserves labels.
Informally, $(f)\sigma_{n, k, r}$ tells us how $f$ is meant to ``rewrite" the ``expressions" that are our labelled generating sets. We see next that these rewrite rules characterise our elements.

\begin{lemma}\label{lem:very_injective_term_action}
Suppose that $k\geq 2$, and $n,r\geq 1$.
If $(L_1,L_2)\in \mathcal{L}_{n,k,r}^I\times \mathcal{L}_{n,k,r}$ and $\im(L_2) \subseteq \im(L_1)$, then there is a unique $f\in \operatorname{End}(\mathbb{F}_{n,k, r})$ such that $(L_1, L_2)\in (f)\sigma_{n, k, r}$.
In particular, if $f, g\in \operatorname{End}(\mathbb{F}_{n,k, r})$ and such a pair $(L_1, L_2)$ is contained in $(f)\sigma_{n, k, r}\cap (g)\sigma_{n, k, r}$, then $f=g$.
\end{lemma}
\begin{proof}
Since $L_1$ is injectively labeled (from $\mathcal{L}_{n,k,r}^I$) and $\im(L_2) \subseteq \im(L_1)$, there is a unique function $f':\operatorname{dom}(L_2)\to \operatorname{dom}(L_1)$ defined by the label preservation condition: $(a)f'$ is the element of $\operatorname{dom}(L_1)$ with the same label as $a \in \operatorname{dom}(L_2)$.
A map $f\in \operatorname{End}(\mathbb{F}_{n,k, r})$ satisfies $(L_1, L_2)\in (f)\sigma_{n, k, r}$ if and only if $f$ restricts to $f'$ on the free generating set $\operatorname{dom}(L_2)$.
As the domain of $L_2$ is a free generating set for $\mathbb{F}_{n,k, r}$, the universal property of free algebras guarantees a unique such endomorphism $f$.
\end{proof}

Among other things, this lemma allows us to define the binary operation for our ``rewrite rules" similar to that for $V$.

\begin{definition}\label{rewritedef}
Suppose that $k\geq 2$, and $n,r\geq 1$.
We define $n\operatorname{Rel}_{k,r}$ to be the monoid whose underlying set is the image of the map $\sigma_{n, k, r}$, i.e., $n\operatorname{Rel}_{k,r} = (\operatorname{End}(\mathbb{F}_{n, k,r}))\sigma_{n, k, r}$.
The multiplication, denoted $\cdot$, is defined such that for $r_1, r_2 \in n\operatorname{Rel}_{k,r}$, their product $r_1 \cdot r_2$ is the unique element of $n\operatorname{Rel}_{k,r}$ which contains the relation composition $r_1 \circ r_2$. That is, if \(r_1\) rewrites an expression and \(r_2\) rewrites the result, then \(r_1\cdot r_2\) does both rewrites.
We show this is well defined in \cref{rewritedefined}.
\end{definition}

\begin{proposition}\label{rewritedefined}
For all $k\geq 2$, and $n,r\geq 1$, the monoid $n\operatorname{Rel}_{k,r}$ from \cref{rewritedef} is well-defined and the map $\sigma_{n,k,r}:\operatorname{End}(\mathbb{F}_{n, k,r}) \to n\operatorname{Rel}_{k,r}$ is an anti-isomorphism.
\end{proposition}
\begin{proof}
The underlying set of \(n\operatorname{Rel}_{k,r}\) is certainly well-defined, and by definition \(\sigma_{n,k,r}:\End(\mathbb{F}_{n, k,r}) \to n\operatorname{Rel}_{k,r}\) is surjective.

It remains to show that this map is injective, the multiplication as described is well-defined, and it reverses the order of  multiplication in \(\End(\mathbb{F}_{n, k,r})\).
If \(f, g\in \End(\mb{F}_{n, k, r})\) and 
\((L_1, L_2)\in (f)\sigma_{n, k, r} \circ (g)\sigma_{n, k, r}\), then there is \(L_3\in \mc{L}_{n,k,r}\) such that \(f\circ L_1\) extends \(L_3\) and \(g\circ L_3\) extends \(L_2\).
Hence \( gf\circ L_1\) extends \(L_2\) and \((L_1, L_2)\in (gf)\sigma_{n, k, r}\). So \((f)\sigma_{n, k, r} \circ (g)\sigma_{n, k, r}\subseteq (gf)\sigma_{n, k, r}\). 

We have shown that for all \(f, g\in \End(\mb{F}_{n, k,r})\), we have 
\[(fg)\sigma_{n, k, r} \supseteq (g)\sigma_{n, k, r} \circ (f)\sigma_{n, k, r}.\]

By \cref{nice_generating_sets}, \((g)\sigma_{n, k, r}\circ (f)\sigma_{n, k, r}\) contains a pair \((L_1,L_2)\) whose first element is from \(\mc{L}_{n,k,r}^I\). Moreover, by Lemma~\ref{lem:very_injective_term_action}, \(fg\) is the only element whose relation contains it.     
\end{proof}

We provide a two dimensional example of working with the rewrite perspective.
\begin{example} \label{example:bakers}
Let $f$ be the element of $2V$ (the group of units of $2M_{2,1}$) which maps the complete prefix code $\{(\mathbf{0},(\varepsilon, 0)), (\mathbf{0},(\varepsilon,1))\}$ to $\{(\mathbf{0},(0, \varepsilon)), (\mathbf{0},(1,\varepsilon))\}$ in this order.
Also let $g$ be the element in $2M_{2, 1}$ which sends $(\mathbf{0},(\varepsilon,\varepsilon))$ to $(\mathbf{0},(0, \varepsilon))$.
Then the corresponding endomorphisms in $\operatorname{End}(\mathbb{F}_{2, 2, 1})$ are defined by the action on the root $\mathbf{0}$:
\[(\mathbf{0})(f)\phi_{n, k, r} = ((\mathbf{0})\alpha_{0, 0}, (\mathbf{0})\alpha_{0,1})\lambda_1, \quad \text{and} \quad (\mathbf{0})(g)\phi_{n, k, r} = (\mathbf{0}) \alpha_{0, 0}.\]
Hence if $L_0, L_1\in \mathcal{L}$ are defined by
\[L_0=\{((\mathbf{0})\alpha_{0, 0}, a), ((\mathbf{0})\alpha_{0, 1}, b)\} \quad \operatorname{ and }\quad L_1=\{((\mathbf{0})\alpha_{1, 0}, a), ((\mathbf{0})\alpha_{1, 1}, b)\}\]
then $(L_0, L_1)\in (f)\phi_{n,k,r}\sigma_{n,k,r}$.

For $r\in n\operatorname{Rel}_{k,r}$, the notation $A\xrightarrow[]{r} B$ means that $A, B \in \mathcal{L}_{n,k,r}$ with $(A, B)\in r$.
We extend this notation to $f\in \operatorname{End}(\mathbb{F}_{n, k, r}) \cup \operatorname{tot}nM_{k, r}$ in the natural fashion via the maps $\phi_{n, k, r}$ and $\sigma_{n, k, r}$ (so that we may simply write $A\xrightarrow[]{f} B$).
Hence the elements of $n\operatorname{Rel}_{k,r}$ corresponding to $f, g$ above can be described as follows.
\[\{((\mathbf{0})\alpha_{0, 0}, a),((\mathbf{0})\alpha_{0, 1}, b)\} \xrightarrow[]{f} \{((\mathbf{0})\alpha_{1, 0}, a),((\mathbf{0})\alpha_{1, 1}, b)\}\]
\[\{((\mathbf{0})\alpha_{0, 0}, a),((\mathbf{0})\alpha_{0, 1}, b)\} \xrightarrow[]{g} \{(\mathbf{0}, a)\} \]

Thus their composite $(fg)\phi_{n,k,r}\sigma_{n,k,r}$ (which corresponds to $g \circ f$ in $\operatorname{End}$) can be described by the composition of relations $f \circ g$:
\begin{align*}
    & \{((\mathbf{0})\alpha_{0, 0}\alpha_{0, 0}, a),
((\mathbf{0})\alpha_{0, 0}\alpha_{0, 1}, b), ((\mathbf{0})\alpha_{0, 1}\alpha_{0, 0}, c), ((\mathbf{0})\alpha_{0, 1}\alpha_{0, 1}, d)\}\\
&\xrightarrow[]{f} \{((\mathbf{0})\alpha_{1, 0}\alpha_{0, 0}, a),
((\mathbf{0})\alpha_{1, 0}\alpha_{0, 1}, b), ((\mathbf{0})\alpha_{1, 1}\alpha_{0, 0}, c), ((\mathbf{0})\alpha_{1, 1}\alpha_{0, 1}, d)\}\\
&\xrightarrow[]{g} \{((\mathbf{0})\alpha_{1, 0}, a), ((\mathbf{0})\alpha_{1, 1}, c)\}.
\end{align*}
The final result is:
\[\{((\mathbf{0})\alpha_{0, 0}\alpha_{0, 0}, a),
((\mathbf{0})\alpha_{0, 0}\alpha_{0, 1}, b), ((\mathbf{0})\alpha_{0, 1}\alpha_{0, 0}, c), ((\mathbf{0})\alpha_{0, 1}\alpha_{0, 1}, d)\} \xrightarrow[]{fg} \{((\mathbf{0})\alpha_{1, 0}, a), ((\mathbf{0})\alpha_{1, 1}, c)\}.\]
Thinking of dimension $0$ as vertical (going down) and dimension $1$ as horizontal (going left to right), these equations can alternatively be represented more succinctly by matrix-like notation:
\[\left(\begin{bmatrix}
a & b
\end{bmatrix}\right) \xrightarrow[]{f} \left(\begin{bmatrix}
a \\
b
\end{bmatrix}\right), \quad \left(\begin{bmatrix}
a & b
\end{bmatrix}\right) \xrightarrow[]{g}
\left(a\right)
 ,\]
\[\left(\begin{bmatrix}\begin{bmatrix}
a & b\end{bmatrix}&
\begin{bmatrix}
c & d
\end{bmatrix}\end{bmatrix}\right) \xrightarrow[]{f} \left(\begin{bmatrix}
a & b\\
c & d
\end{bmatrix}\right)\xrightarrow[]{g} \left(\begin{bmatrix}
a \\
c
\end{bmatrix}\right) ,\]
\[\left(\begin{bmatrix}\begin{bmatrix}
a & b\end{bmatrix}&
\begin{bmatrix}
c & d
\end{bmatrix}\end{bmatrix}\right) \xrightarrow[]{fg}  \left(\begin{bmatrix}
a \\
c
\end{bmatrix}\right) .\]
The curved brackets $\left(\ldots\right)$ allow for introducing multiple matrices in the case that $r>1$. For example with $r=n=3$ and $k=2$ we can denote
\[\{ ((\mathbf{0})\alpha_{0,1},a), ((\mathbf{1})\alpha_{0,1},b), (\mathbf{1},c), (\mathbf{2},d) \} \xrightarrow[]{h} \{ (\mathbf{0},a), (\mathbf{1},d), ((\mathbf{2})\alpha_{1,1},b), ((\mathbf{2})\alpha_{1,0},c) \}\]
as
\[\left(\begin{bmatrix}
a \\
b
\end{bmatrix},c,d \right)\xrightarrow[]{h}\left(a,d,\begin{bmatrix}
c & b
\end{bmatrix}\right).\]
The matrix notation here is significantly easier to read than fully written out free generating sets, and we will use this notation going forward where possible. We restrict our visual examples to the first two dimensions, $0$ and $1$, as representing higher dimensional expansions visually is challenging.
\end{example}

In the following definitions we define how we can ``defer" the action of a continuous map or rewrite rule $f$. Informally this works by making the map treat each shrubbery as it previously treated the corresponding root.

\begin{definition}[Root Systems and Deferments]\label{defermentdef}
Suppose that $n,r\geq 1$ and $k\geq 2$.
A \textit{root system} $s=(s_0, s_1, \ldots, s_{r-1})$ is an element of $(nT_{k,r})^r$ such that the cones of the entries of $s$ are pairwise disjoint. 
For each $f\in \operatorname{tot} nM_{k, r}$, let $f_{s} \in \operatorname{tot} nM_{k,r}$ be the element defined by
\[(s_ix)f_s=s_{((\mathbf{i},x)f)^{\rt}}\cdot ((\mathbf{i},x)f)^{\shrub} \quad\text{ and }\quad (y)f_s=y\]
for all $i<r$, $y\in \mathfrak{C}_{n,k, r}\backslash \bigcup_{i<r}w_i\mathfrak{C}_{n,k}$ and $x\in \mathfrak{C}_{n,k}$.
For $f\in n\operatorname{Rel}_{k,r}$, we define $f_{s}\in n\operatorname{Rel}_{k,r}$ analogously via the anti-isomorphism $\sigma_{n,k,r}$.
We call such elements the \textit{deferments} of $f$.
\end{definition}
\begin{example}\label{numberlabelexample}
If $f$ is as in \cref{example:bakers} (recall that the choice of label set is unimportant) with
\[\left(\begin{bmatrix}
0 & 1
\end{bmatrix}\right) \xrightarrow[]{f} \left(\begin{bmatrix}
0 \\
1
\end{bmatrix}\right)\]
and $s=(\mathbf{0},(0, 01))$ then
\[\left(\begin{bmatrix}
\begin{bmatrix} 0\\  \begin{bmatrix} 1&  2\end{bmatrix}\end{bmatrix}& 3\\
4 & 5
\end{bmatrix}\right) \xrightarrow[]{f_{s}} \left(\begin{bmatrix}
\begin{bmatrix} 0\\  \begin{bmatrix} 1\\  2\end{bmatrix}\end{bmatrix}& 3\\
4 & 5
\end{bmatrix}\right).\]
\end{example}

\section{Finite Presentability}

\subsection{Proving presentability}
The preliminary setup is now complete, and we can begin working towards a proof of finite presentability in earnest, beginning with finite generation.
\begin{lemma}\label{lem:gens}
Let \(n,r\geq 1\) and \(k\geq 2\). The group of units of the monoid \(n\operatorname{Rel}_{k,r}\) is finitely generated. 
Moreover, the monoid \(n\operatorname{Rel}_{k,r}\) is generated by its group of units together with the elements \(U\), \(\pi^{(\mathbf{0},(0,\varepsilon,\ldots))}\) defined by 
\[(0, 1,\ldots,r-1) \xrightarrow[]{U} \left(\begin{bmatrix}
0\\
\vdots\\
0
\end{bmatrix}, 1,2,\ldots,r-1\right),\]

\[ \left(\begin{bmatrix}
0\\
\vdots\\
k-1
\end{bmatrix}, k, k+1,\ldots, k+r-2\right) \xrightarrow[]{\pi^{(\mathbf{0},(0,\varepsilon,\ldots))}} (0, k, k+1,\ldots, k+r-2)\]
where we are using the notation of \cref{example:bakers} and \cref{numberlabelexample}.
\end{lemma}

\begin{proof}
The proof is structured as follows. We first show that the group \(\operatorname{Aut}(\mathbb{F}_{n,k,r})\) is finitely generated.  \cref{rewritedefined} then implies the first part of the lemma. We then conclude the proof working in \(n\operatorname{Rel}_{k,r}\).

Note that by \cref{lem:FreeGenExpansion}, we have \(\mathbb{F}_{n,k,r}\cong \mathbb{F}_{n,k,r+k-1}\) so when proving that \(\operatorname{Aut}(\mathbb{F}_{n,k,r})\) is finitely generated we can assume without loss of generality that \(r>1\).
\cref{nice_generating_sets} implies that each element of  \(\operatorname{Aut}(\mathbb{F}_{n,k,r})\) can be defined by mapping a root expansion generating set bijectively to another root expansion generating set.
By \cref{AntichainImisGenSet}, each free generating set for \(\mathbb{F}_{n,k,r}\) has cardinality in \(r+(k-1)\N\). Moreover, for all \(m\in \N\), the set
\[\{\boldsymbol{0},\boldsymbol{1},\ldots,\mathbf{r-2}\}\cup \left(\bigcup_{i<m}\{(\mathbf{r-1})\alpha_{0,k-1}^i\alpha_{0,0},\ldots, (\mathbf{r-1})\alpha_{0,k-1}^i\alpha_{0,k-2}\}\right)\cup \{(\mathbf{r-1})\alpha_{0,k-1}^{m}\}\]
is a root expansion generating set with cardinality \(r+(k-1)m\). 
We call these ``right-heavy" free generating sets.
Thus to generate \(\operatorname{Aut}(\mathbb{F}_{n,k,r})\), it suffices to be able to permute the elements of right-heavy free generating sets and map every root expansion generating set to one of these.

Let \(G\) denote the set of root expansion generating sets which require at most two elementary expansions (as in \cref{rootexpansions}). In particular \(G\) is finite. Let \(\Sigma\) be the set of all elements of \(\operatorname{Aut}(\mathbb{F}_{n,k,r})\) which define a bijection from one element of \(G\) to another.

We now give an algorithm which uses elements of \(G\) to convert a root expansion generating set \(X\) into a right-heavy free generating set.

\begin{enumerate}
    \item[$\substack{\text{Step }\mathbf{i}\\\text{ for }\mathbf{i}< r}$] If \(\mathbf{i}\notin X\), then there is \(l\) such that in the construction of \(X\), we perform an expansion about \(\mathbf{i}\) in dimension \(l\).
    Choose \(\mathbf{j}<r\) distinct from \(\mathbf{i}\) (recall that we assumed \(r>1\)).
    Let \(g\in G\) be an element inducing a bijection between the free generating sets \[(\{\mathbf{0},\ldots,\mathbf{r-1}\}\backslash\{\mathbf{i}\})\cup \{\mathbf{i}\alpha_{l,0},\ldots, \mathbf{i}\alpha_{l,k-1}\},\] 
    \[(\{\mathbf{0},\ldots,\mathbf{r-1}\}\backslash\{\mathbf{j}\})\cup \{\mathbf{j}\alpha_{0,0},\ldots, \mathbf{j}\alpha_{0,k-1}\}.\] 
    By the choice of \(l\), it follows that \((X)g\) is a root expansion generating set as well. Replace \(X\) with \((X)g\). 
    The number of \(\alpha\)-type operations needed to construct the elements written with the free generator \(\mathbf{i}\) has decreased. Moreover if \(l\neq 0\), then the number of \(\alpha_{l,m}\) appearing the construction of \(X\) for all \(m<k\) has decreased as well.
    Repeat this step until \(\mathbf{i}\in X\).
    \item[Step r] Observe that \(\mathbf{r-1}\in X\) and \(X\) is a root expansion generating set which only requires expansions in dimension \(0\).
    If there is some \(\mathbf{i}<\mathbf{r-1}\) with \(\mathbf{i}\notin X\), then choose one of them and let \(g\in G\) be an element inducing a bijection between the free generating sets \[(\{\mathbf{0},\ldots,\mathbf{r-1}\}\backslash\{\mathbf{i}\})\cup \{\mathbf{i}\alpha_{0,0},\ldots, \mathbf{i}\alpha_{0,k-1}\},\] 
    \[\{\mathbf{0},\ldots,\mathbf{r-2}\}\cup \{(\mathbf{r-1})\alpha_{0,0},\ldots, (\mathbf{r-1})\alpha_{0,k-1}\}.\] 
     Replace \(X\) with \((X)g\). If \(\mathbf{i}\) does not exist then \(X=\{\mathbf{0},\ldots,\mathbf{r-1}\}\) so we are done.
    \item[Step r+1] At this point \(X\) is a root expansion generating set which requires only expansions in dimension \(0\) and \(\{(\mathbf{r-1})\alpha_{0,0},\ldots, (\mathbf{r-1})\alpha_{0,k-1}\} \subseteq X\).
    While there is some \(\mathbf{i}<\mathbf{r-1}\) with \(\mathbf{i}\notin X\), choose one of them and let \(g\in G\) be an element inducing a bijection between the free generating sets 
    \[(\{\mathbf{0},\ldots,\mathbf{r-2}\}\backslash\{\mathbf{i}\})\cup \{\mathbf{i}\alpha_{0,0},\ldots, \mathbf{i}\alpha_{0,k-1}\} \cup \{(\mathbf{r-1})\alpha_{0,0},\ldots, (\mathbf{r-1})\alpha_{0,k-1}\},\] 
    \[\{\mathbf{0},\ldots,\mathbf{r-2}\}\cup \{(\mathbf{r-1})\alpha_{0,0},\ldots, (\mathbf{r-1})\alpha_{0,k-2}\}\cup \{(\mathbf{r-1})\alpha_{0,k-1}\alpha_{0,0},\ldots, (\mathbf{r-1})\alpha_{0,k-1}\alpha_{0,k-1}\}\] 
     where we specifically map \((\mathbf{r-1})\alpha_{0,k-1}\) to \((\mathbf{r-1})\alpha_{0,k-1}\alpha_{0,k-1}\).
     Note that applying \(g\) to \(X\) will always leave the elements of \(X\) using \(\mathbf{r-1}\) in the form required by a right-heavy generating set and decreases the number of elements of \(X\) not in the required form.
     Replace \(X\) with \((X)g\).
     When such an  \(\mathbf{i}\) no longer exists, we must have that \(X\) is right-heavy.
\end{enumerate}

We now know that the group generated by \(G\) can transform any root expansion generating set to a right-heavy generating set. The elements of \(G\) can also, by definition, permute the generating set
\[\{\boldsymbol{0},\boldsymbol{1},\ldots,\mathbf{r-2}\}\cup\{(\mathbf{r-1})\alpha_{0,0},\ldots, (\mathbf{r-1})\alpha_{0,k-2}\} \cup \{(\mathbf{r-1})\alpha_{0,k-1}\alpha_{0,0},\ldots,(\mathbf{r-1})\alpha_{0,k-1}\alpha_{0,k-1}\}\]
in all possible ways.
As we can also use \(G\) to map \(\mathbf{r-1}\) to \((\mathbf{r-1})\alpha_{0,k-1}\), it follows by conjugation that there are elements of \(\langle G\rangle\) which permute the set
\[\{(\mathbf{r-1})\alpha_{0,k-1}^m\alpha_{0,0},\ldots, (\mathbf{r-1})\alpha_{0,k-1}^m\alpha_{0,k-2}\}\cup\{(\mathbf{r-1})\alpha_{0,k-1}^{m+1}\alpha_{0,0},\ldots,(\mathbf{r-1})\alpha_{0,k-1}^{m+1}\alpha_{0,k-1}\}\]
for all \(m\geq 0\) in all possible ways while fixing the other elements of the corresponding right heavy generating set. Together these allow us to generate all permutations of all right-heavy generating sets. This concludes the proof that \(\operatorname{Aut}(\mathbb{F}_{n,k,r})\) is finitely generated.

We now show that \(n\operatorname{Rel}_{k,r}\) is generated by its group of units together with the elements \(U\) and \(\pi^{(\mathbf{0},(0,\varepsilon,\ldots))}\).
Suppose that \(f\in n\operatorname{Rel}_{k,r}\) is defined by the rewrite rule \(L_0 \xrightarrow[]{f} L_1\). Let \(g_0\) and \(g_1\) be invertable elements such that there are labelled right-heavy generating sets \(L_0'\) and \(L_1'\) with 
\[L_0'\xrightarrow[]{g_0} L_0 \xrightarrow[]{f} L_1\xrightarrow[]{g_1} L_1'.\]
Using invertible elements, we can permute the labels of free generating sets however we wish. Moreover by applying \(U\) we can create more copies of the label of \(\mathbf{0}\) from a given free generating set. Thus there is \(U'\) in the monoid generated by \(U\) and the units of  \(n\operatorname{Rel}_{k,r}\) such that \(L_0' \xrightarrow[]{U'} L_0''\), \(L_0''\) is right-heavy and all labels of \(L_1'\) appear more often in \(L_0''\).

Recall that by \cref{AntichainImisGenSet}, the multiset of labels of \(L_0''\) and  \(L_1'\) have sizes in \(r+(k-1)\N\).
By applying \(\pi^{(\mathbf{0},(0,\varepsilon,\ldots))}\), we can delete \(k-1\) labels from a labelled free generating set of the correct shape.
Thus there is \(\pi'\) in the monoid generated by \(\pi^{(\mathbf{0},(0,\varepsilon,\ldots))}\) and units such that \(L_0'' \xrightarrow[]{\pi'} L_0'''\) where \(L_0'''\) is right-heavy and has the same multiset of labels as \(L_1'\). Thus there is a invertible element \(g\) with \(L_0''' \xrightarrow[]{g} L_1'\). Together we get
\[L_0 \xrightarrow[]{g_0^{-1}} L_0'\xrightarrow[]{U'} L_0'' \xrightarrow[]{\pi'} L_0'''\xrightarrow[]{g} L_1'\xrightarrow[]{g_1^{-1}} L_1\]
so \(f=g_0^{-1}U'\pi' gg_1^{-1}\).
As \(g_0^{-1}U'\pi' gg_1^{-1}\)
belongs to the monoid generated by the group of units of \(n\operatorname{Rel}_{k,r}\) together with the elements \(U\) and \(\pi^{(\mathbf{0},(0,\varepsilon,\ldots))}\), the result follows.
\end{proof}

With generation done, we now move on to relations. We denote the set of words over the alphabet \(A\) by \(A^*\), and we will show that the congruence on the free monoid \(A^*\) defined by the natural quotient map onto \(n\operatorname{Rel}_{k,r}\) is generated by a finite set of relations.

The strategy is to define a nice submonoid \(nP_{k,1}\) of \(n\operatorname{Rel}_{k,1}\) and use the interaction between elements of \(n\operatorname{Rel}_{k,1}\) and \(nP_{k,1}\) to determine when they are equal.
\begin{definition}[Embedded Product of Free Monoids]
Let \(n\geq 1\) and \(k\geq 2\).
 We denote by \(nP_{k,1}\) the subset of \(n\text{Rel}_{k,1}\) consisting of those elements which can be expressed in the form \(L_0 \rightarrow (0)\).
\end{definition}

\begin{proposition}\label{pstructure}
Let \(n\geq 1\) and \(k\geq 2\). The set \(nP_{k,1}\) is a monoid and there is an isomorphism \(\chi_{n,k}:nP_{k,1}\to (A_k^*)^n\) such that
\[(\mathbf{0},(f)\chi_{n,k})\operatorname{tree}_{n,k,1}=((\mathbf{0})(f)\sigma_{n,k,r}^{-1}).\]
\end{proposition}
\begin{proof}
  
    Note that \(nP_{k,r}\sigma_{n,k,1}^{-1}\) is the set of elements of \(\End(\mathbb{F}_{n,k,1})\) which map the free generator to an element of \(\im(\operatorname{tree}_{n,k,1})\).
    Each such element is obtained from the free generator \(\mathbf{0}\) by a sequence of unary operations \(\alpha_{i,j}\).
   As seen in \cref{free_understanding}, the unary operations  \(\alpha_{i,j}\) generate a copy of the monoid \(\oplus_{i<n} \makeset{\alpha_{i,j}}{$j<k$}^*\cong (A_k^*)^n\). 
   Suppose that \(\alpha,\beta\in \oplus_{i<n} \makeset{\alpha_{i,j}}{$j<k$}^*\cong (A_k^*)^n\), and \(f,g,\in \End(\mathbb{F}_{n,k,1})\) satisfy 
   \[(\mathbf{0})f=(\mathbf{0})\alpha \text{ and }(\mathbf{0})g=(\mathbf{0})\beta.\]
   As \(f,g\) are endomorphisms, it follows that
   \[(\mathbf{0})fg=((\mathbf{0})\alpha)g=(\mathbf{0}g)\alpha=(\mathbf{0})\beta\alpha.\]
   Thus \(nP_{k,r}\sigma_{n,k,1}^{-1}\) is anti-isomorphic to \(\oplus_{i<n} \makeset{\alpha_{i,j}}{$j<k$}^*\cong (A_k^*)^n\). As \(\sigma\) is an anti-isomorphism the result follows.
\end{proof}

We need to use these elements to identify elements of \(n\operatorname{Rel}_{k,1}\). The following proposition allows us to do this.
\begin{proposition}\label{determinewithprjections}
   Let \(n\geq 1\) and \(k\geq 2\). For all \(f\in n\operatorname{Rel}_{k,1}\) there is a \(N_f\in \N\) such that for all \(N\geq N_f\) we have:
   \begin{enumerate}
       \item  for all \(p\in ((A_k^N)^n)\chi_{n,k}^{-1}\) we have \(fp\in nP_{k,1}\),
       \item if \(g\in n\operatorname{Rel}_{k,1}\) and \(gp=fp\) for all \(p\in ((A_k^N)^n)\chi_{n,k}^{-1}\), then \(f=g\).
   \end{enumerate}
  
\end{proposition}
\begin{proof}
Note that the set \(((A_k^N)^n)\chi_{n,k}^{-1}\sigma_{n,k,1}^{-1} \phi_{n,k,1}\) is the subset of \(\tot n M_{k,1}\) consisting of elements with take the prefix code \(\{(\mathbf{0},(\varepsilon,\ldots,\varepsilon))\}\) to an element of \(X_1\times (A_k^N)^n\).

We show the analogous claim holds in \(\tot n M_{k,r}\) which is sufficient by the anti-isomorphisms \(\sigma_{n,k,r}\circ \phi_{n,k,r}^{-1}\).
Let \(f\in \tot n M_{k,r}\) and let \(N_f\) be large enough that \(f\) can be defined by a pseudotree pair whose tree has depth less than \(N_f\).
It follows that if \(p\in \tot n M_{k,1}\) takes the prefix code \(\{(\mathbf{0},(\varepsilon,\ldots,\varepsilon))\}\) to an element of \(X_1\times (A_k^N)^n\), then \(pf\) can also be defined by a psuedotree pair with domain prefix code \(\{(\mathbf{0},(\varepsilon,\ldots,\varepsilon))\}\).

Moreover if \(pg=pf\) for all of the \(p\) which take the prefix code \(\{(\mathbf{0},(\varepsilon,\ldots,\varepsilon))\}\) to an element of \(X_1\times (A_k^N)^n\), then \(g\) and \(f\) agree at all points in their domain so are equal.
\end{proof}

We now start to build our presentation. There are many things we need to define first.
\begin{definition}[The presentation]\label{presentationdef}
    Let \(n\geq 1\) and \(k\geq 2\).
    \begin{enumerate}
        \item Let \(\Sigma\) be a finite generating set for \(n\operatorname{Rel}_{k, 1}\) containing the inverses of all invertible elements (which exists from \cref{lem:gens}),
        \item Let \(\langle \Lambda | R_\Lambda\rangle\) be a finite presentation for the monoid \(nP_{k,1}\) (which exists by \(\cref{pstructure}\)). We choose, for each element of $\Lambda$, a single string in \(\Sigma^*\) representing it in \(n\operatorname{Rel}_{k, 1}\).
        We then identify each element of \(\Lambda\) with its corresponding string and consider \(R_\Lambda\) as a set of relations of \(n\operatorname{Rel}_{k,1}\),
        \item For each shrub \(s\in (A_{k}^*)^n\cong nP_{k,1}\) let \(\pi^s\in \Lambda^*\) be a string representing the corresponding element in \(nP_{k,1}\),
        \item Let \(\Phi: \Sigma^*\to n\operatorname{Rel}_{k,1}\) be the homomorphism mapping each element of \(\Sigma \) to the element it represents;
        \item For each \(x\in \Sigma\), let \((x)\operatorname{SDef}\) be a fixed string in \(\Sigma^*\) with represents the deferment of \(x\in \Sigma\) to the root system \(((\boldsymbol{0},(1,\varepsilon,\ldots,\varepsilon)))\in nT_{k,1}\) (see \cref{defermentdef}).
        Extend \(\operatorname{SDef}\) to a map from \(\Sigma^*\) via the universal property of free monoids.
        \item For all elements \(s,t\in nT_{k,1}\backslash \{(\boldsymbol{0},(\varepsilon,\ldots,\varepsilon))\}\), let \(p_{s,t}\) be a string in the invertible elements of \(\Sigma\) representing an element which maps the root system \((s)\) to the root system \((t)\).
        We then define \(\operatorname{Def}_s:(\Sigma\cup \Lambda) \to (\Sigma\cup \Lambda)^*\) by \((x)\operatorname{Def}_s=p_{(\boldsymbol{0},(1,\varepsilon,\ldots,\varepsilon)),s}^{-1}(x)\operatorname{SDef} p_{(\boldsymbol{0},(1,\varepsilon,\ldots,\varepsilon)),s}\) (here inverting a string means to write it backwards and replace each generator with its inverse). 
        \item Let \(Y\in \Sigma^*\) be a string representing the element \(\{(\boldsymbol{0},a)\}\to \{(\mathbf{0}\alpha_{i,j},a)|i<n,j<k\}\);
        \item Let \(R_0\) be the union of the following sets (where \(N_x\) is from \cref{determinewithprjections})
        \begin{enumerate}
    \item \(\makeset{(ab, \varepsilon), (ba, \varepsilon)\in \Sigma^* \times \Sigma^*}{\(a, b\in \Sigma\) are inverses in the group of units of \(n\operatorname{Rel}_{k, 1}\)}\).
        \item \(\makeset{((x)\operatorname{SDef}\operatorname{SDef}, (x)\operatorname{Def}_{(\boldsymbol{0},(11,\varepsilon,\ldots,\varepsilon))})\in \Sigma^* \times \Sigma^*}{\(x\in \Sigma\)}\).
           \item 
    \(\makeset{\left(x ,Y\prod_{s\in (A_k^1)^n}(x\pi^{s})\operatorname{Def}_{(\mathbf{0},s)}\right)\in \Sigma^*\times \Sigma^*}{\(x\in \Sigma\)}\)  where the product over the elements of \((A_k^1)^n\) is taken in any fixed order.
     \item 
    \(\makeset{\left(xY ,Y \prod_{s\in (A_k^1)^n}(x)\operatorname{Def}_{(\mathbf{0},s)}\right)\in \Sigma^*\times \Sigma^*}{\(x\in \Sigma\)}\)  where the product over the elements of \((A_k^1)^n\) is taken in any fixed order.
      
     \item \(\makeset{((x)\operatorname{Def}_{v} (y)\operatorname{Def}_{w}, (y)\operatorname{Def}_{w} (x)\operatorname{Def}_{v})\in \Sigma^* \times \Sigma^*}{ \(x,y\in \Sigma\), \(v, w\in X_1\times (A_k^1)^n\), \(v\neq w\)}\).
    \item \(\makeset{(x\pi^{s_{x,1}},\pi^{s_{x,2}})}{ where for each \(x\in \Sigma\) and \(s_{x,1}\in (A_k^{N_x})^n\), we fix one choice of\\ \(s_{x,2}\in (A_k^*)^n\) such that
    \((x\pi^{s_{x,1}})\Phi=(\pi^{s_{x,2}})\Phi\)}\).

\end{enumerate}
\item Let \(R_1=R_0\cup R_\Lambda\), and \(R=R_1 \cup \makeset{((x)\operatorname{SDef}, (y)\operatorname{SDef})}{\((x, y)\in R_1\)}\)
    \end{enumerate}
Let \(n\operatorname{PRel}_{k}\) be the monoid presented by \(\langle \Sigma  | R\rangle\).
\end{definition}

\begin{lemma}\label{well-defs}
    Let \(n\geq 1\) and \(k\geq 2\), and  \(s\in X_1\times (A_k^*)^n\).
    The map \(\operatorname{Def}_s:\Sigma \to \Sigma^*\) from \cref{presentationdef} induces a well defined endomorphism of the monoid \(n\operatorname{PRel}_{k,1}\).
\end{lemma}
\begin{proof}
    We check that \(\operatorname{Def}_s\) is a well-defined endomorphism.
    Note that for all \((x, y)\in R_1\), we have \(((x)\operatorname{SDef},(y)\operatorname{SDef})\in R\) and for all other \((x,y)\in R\backslash R_1\) there are \((x',y')\in R_1\) with
    \(((x')\operatorname{SDef}, (y')\operatorname{SDef})=(x,y)\).
    If \((x',y')=(x'_0,\ldots x_{j-1}'.y'_0,\ldots y_{j-1}')\), then we have (by the relations in sets (a) and (b)) that
    \begin{align*}
(x)\operatorname{SDef}&=(x')\operatorname{SDef}\operatorname{SDef}\\
&=\prod_{i<j}(x_i')\operatorname{SDef}\operatorname{SDef}\\
        &=_{n\operatorname{PRel}_{k}}\prod_{i<j}(x_i')\operatorname{Def}_{(\boldsymbol{0},(11,\varepsilon,\ldots,\varepsilon))}\\
        &= \prod_{i<j} p_{(\boldsymbol{0},(1,\varepsilon,\ldots,\varepsilon)),s}^{-1}(x_i')\operatorname{SDef} p_{(\boldsymbol{0},(1,\varepsilon,\ldots,\varepsilon)),s}\\ &=_{n\operatorname{PRel}_{k}} p_{(\boldsymbol{0},(1,\varepsilon,\ldots,\varepsilon)),s}^{-1}\left(\prod_{i<j} (x_i')\operatorname{SDef}\right) p_{(\boldsymbol{0},(1,\varepsilon,\ldots,\varepsilon)),s}\\
&= p_{(\boldsymbol{0},(1,\varepsilon,\ldots,\varepsilon)),s}^{-1}\left( (x')\operatorname{SDef}\right) p_{(\boldsymbol{0},(1,\varepsilon,\ldots,\varepsilon)),s}\\
&= p_{(\boldsymbol{0},(1,\varepsilon,\ldots,\varepsilon)),s}^{-1}\left( (y')\operatorname{SDef}\right) p_{(\boldsymbol{0},(1,\varepsilon,\ldots,\varepsilon)),s}\\
&= _{n\operatorname{PRel}_{k}} (y)\operatorname{SDef}.
    \end{align*}
So \((x)\operatorname{SDef}\) and  \((y)\operatorname{SDef}\) are equal in \({n\operatorname{PRel}_{k}}\).
Thus the map \(\operatorname{SDef}:\Sigma^*\to \Sigma^*\) induces a well defined endomorphism of the monoid \({n\operatorname{PRel}_{k}}\).
As conjugation by a unit is an automorphism, it follows (from the definition of \(\operatorname{Def}_w\)) that for all non-trivial root systems \(w\), the map \(\operatorname{Def}_w:\Sigma^* \to \Sigma^*\) also induces an endomorphism of the monoid \({n\operatorname{PRel}_{k}}\).
\end{proof}
\begin{lemma}\label{undefer}
    Let \(n\geq 1\) and \(k\geq 2\).
    For all \(x\in n\operatorname{PRel}_{k}\) we have
    \[x=_{n\operatorname{PRel}_{k}} Y \prod_{s\in  (A_{k}^1)^n}(x\pi^{s})\operatorname{Def}_{(\mathbf{0},s)}\]
    Using the deferment endomorphisms from \cref{well-defs}.
\end{lemma}
\begin{proof}
    First note that by \cref{well-defs} and the relations in set (e), if \(v, w\in X_1\times (A_k^1)^n\) are distinct then
    \begin{equation}\label{comm-eq}\tag{$*$}
        \text{ all elements of }\im(\operatorname{Def}_{s})\leq n\operatorname{PRel}_{k}\text{ commute with all elements of }\im(\operatorname{Def}_{t})\leq n\operatorname{PRel}_{k}.
    \end{equation}
  
    Let \(x=x_0x_1\ldots x_{j-1}\in \Sigma^*\). By the relations in set (c) from \cref{presentationdef}, we have for all \(i< j\) that 
\[x_i=_{n\operatorname{PRel}_{k}} Y \prod_{s\in  (A_{k}^1)^n}(x_i\pi^{s})\operatorname{Def}_{(\mathbf{0},s)}.\]
 Thus
\begin{align*}
x&=_{n\operatorname{PRel}_{k}}  x_0x_1\dots x_{j-2}Y\prod_{s\in (A_k^1)^n}(x_{j-1}\pi^{s})\operatorname{Def}_{(\mathbf{0},s)}\\
  &=_{n\operatorname{PRel}_{k}}  x_0x_1\dots x_{j-3}Y\left(\prod_{s\in (A_k^1)^n}(x_{j-2})\operatorname{Def}_{(\mathbf{0},s)}\right)\left(\prod_{s\in (A_k^1)^n}(x_{j-1}\pi^{s})\operatorname{Def}_{(\mathbf{0},s)}\right)\quad \substack{\text{ using the relations}\\ \text{ in set (d)}}\\ &=_{n\operatorname{PRel}_{k}}  Y \left(\prod_{i<j-1}\prod_{s\in (A_k^1)^n}(x_{i})\operatorname{Def}_{(\mathbf{0},s)}\right)\left(\prod_{s\in (A_k^1)^n}(x_{j-1}\pi^{s})\operatorname{Def}_{(\mathbf{0},s)}\right)\quad \substack{\text{ using the relations}\\ \text{ in set (d)}}\\
   &=_{n\operatorname{PRel}_{k}}  Y \prod_{s\in (A_k^1)^n}\left(\left(\prod_{i<j-1} (x_i)\operatorname{Def}_{(\mathbf{0},s)}\right)(x_{j-1}\pi^{s})\operatorname{Def}_{(\mathbf{0},s)}\right)\quad \text{ by }(*) \\ 
&=_{n\operatorname{PRel}_{k}}  Y \prod_{s\in (A_k^1)^n}(x_0x_1\ldots x_{j-1}\pi^{s})\operatorname{Def}_{(\mathbf{0},s)}\quad \text{ by \cref{well-defs}} \\
&=_{n\operatorname{PRel}_{k}}  Y \prod_{s\in (A_k^1)^n}(x\pi^{s})\operatorname{Def}_{(\mathbf{0},s)}\quad \text{ by the definition of }x .
\end{align*}

\end{proof}
\begin{lemma}\label{lem:pres}
Let \(n\geq 1\) and \(k\geq 2\). The monoid \(n\operatorname{Rel}_{k,1}\) is isomorphic to the finitely presented monoid \(n\operatorname{PRel}_{k}\) from \cref{presentationdef}.
\end{lemma}

\begin{proof}
Note that \(n\operatorname{PRel}_{k}\) is generated by \(\Sigma\) and  \(n\operatorname{Rel}_{k,1}\) is generated by \((\Sigma)\Phi\). 
It is routine to verify that the relations in the presentation for \(n\operatorname{PRel}_{k}\) hold in the monoid \(n\operatorname{Rel}_{k,1}\).
Suppose that \(g,h\in \Sigma^*\) are arbitrary with \((g)\Phi=(h)\Phi\). It suffices to show that \(g\) and \(h\) are equal in \(n\operatorname{PRel}_{k}\).

Using the relations in set (f) from \cref{presentationdef}, it follows that there is \(N\in \N\) such that for all \(p\in (A_k^N)^n\) we have \(g\pi^p,h\pi^p\in \langle \Lambda \rangle\).
As we have a presentation for the monoid generated by \(\Lambda\), it follows from the equality \((g)\Phi=(h)\Phi\) that \(g\pi^p=_{n\operatorname{PRel}_{k}} h\pi^p\) for all \(p\in (A_k^N)^n\).

We need only show that the equalities \(g\pi^p=_{n\operatorname{PRel}_{k}} h\pi^p\) for all \(p\in (A_k^N)^n\) together imply that \(g=_{n\operatorname{PRel}_{k}} h\). We show this by induction on \(N\). The base case of \(N=0\) is immediate.

We show the statement follows from the inductive hypothesis:
\begin{itemize}
    \item the equalities \(g\pi^p=_{n\operatorname{PRel}_{k}} h\pi^p\) for all \(p\in (A_k^{N-1})^n\) together imply that \(g=_{n\operatorname{PRel}_{k}} h\).
\end{itemize}

Note that \((A_k^{N-1})^n(A_k^{1})^n=(A_k^{N})^n\). 
So for all \(s\in (A_k^{N-1})^n\) and \(s'\in (A_k^{1})^n\) we have
\[g\pi^s\pi^{s'}=_{n\operatorname{PRel}_{k}} h\pi^s\pi^{s'}.\]
From this it follows that for all \(s\in (A_k^{N-1})^n\) we have
\[Y \prod_{s'\in (A_k^1)^n} (g\pi^s \pi^{s'})\operatorname{Def}_{(\mathbf{0},s')}=_{n\operatorname{PRel}_{k}} Y \prod_{s'\in (A_k^1)^n} (h\pi^s \pi^{s'})\operatorname{Def}_{(\mathbf{0},s')}.\]

Thus, by \cref{undefer}, for all \(s\in (A_k^{N-1})^n\) we have \(g\pi^s =_{n\operatorname{PRel}_{k}} h\pi^s\). The result then follows from the inductive hypothesis.
\end{proof}
\begin{corollary}\label{finishr=1}
    Let \(n\geq 1\) and \(k\geq 2\).
The monoid \(\operatorname{tot}nM_{k,1}\) is finitely presented.
\end{corollary}

\begin{proof}
    This follows from Lemma~\ref{lem:pres}, together with the fact that \(\phi_{n,k,r}\sigma_{n,k,1}:\operatorname{tot}nM_{k, 1}\to n\operatorname{Rel}_{k, 1}\) is an anti-isomorphism.
\end{proof}

We can now prove \cref{main2}.
\begin{theorem}\label{done}
Let \(n,r\geq 1\) and \(k\geq 2\).
The monoid \(\operatorname{tot}nM_{k,r}\) is finitely presented.
\end{theorem}
\begin{proof}
By partitioning \(\mathfrak{C}_{n,k}\) into \(\geq r\) clopen sets, we may fix some injective \(f\in \operatorname{tot}nM_{k,r}\) whose image is contained in \(\{\mathbf{0}\}\times \mathfrak{C}_{n,k}\). Let \(g\in \operatorname{tot}nM_{k,r}\) be a fixed element which agrees with \(f^{-1}\) on the image of \(f\).
In particular \(fg\) is the identity function. 

For each clopen set \(U\subseteq \mathfrak{C}_{n,k,r}\), we define the monoid \(M_U\) by 
\[M_U:= \makeset{h\in \operatorname{tot}nM_{k,r}}{\((U)h\subseteq U\) and \((x)h=x\) for all \(x\notin U\)}.\]
As each clopen set \(U\) is compact and open, there is always a \(m\in \N\) and a set of shrubberies \(S_U\subseteq X_r\times A_k^m\) such that \(U=\cup_{w\in S_U} w\mathfrak{C}_{n,r}\). In particular, \(M_U\) is isomorphic to \(\operatorname{tot}nM_{k,|S_U|}\).
In particular, \(M_{\{\mathbf{0}\}\times \mathfrak{C}_{n,k}}\) is isomorphic to \(\operatorname{tot}nM_{k,1}\).
By \cref{finishr=1}, let \(\langle \Sigma | R\rangle\) be a finite presentation for \(M_{\{\mathbf{0}\}\times \mathfrak{C}_{n,k}}\).

As \(f\) is injective, the submonoid \(M_{\im(f)}\) of \(M_{\{\mathbf{0}\}\times \mathfrak{C}_{n,k}}\) is isomorphic to \(\operatorname{tot}nM_{k,r}\) via the isomorphism \(\phi\) defined by
\((x)\phi = f x g.\) 
By \cref{lem:gens}, \cref{rewritedefined} and \cref{psidefined}, the monoid \(\operatorname{tot}nM_{k,r}\) is finitely generated.
Thus there is a finite subset \(\Delta\) of \(\Sigma^*\) such that \(M_{\im(f)} \cong \operatorname{tot}nM_{k,r}\) is the set of elements of \(M_{\{\mathbf{0}\}\times \mathfrak{C}_{n,k}}\) which can be represented as products of strings from \(\Delta\).

We write \(\Delta^*\) for the submonoid of \(\Sigma^*\) generated by \(\Delta\).
As \((\im(f))f\subseteq \im(f)\), we can fix some \(f'\in  \Delta^*\) representing the element which agrees with \(f\) on the set \(\im(f)\) and fixes all other points.
Note that \(ff'\) and \(ff\) represent the same element of \(\operatorname{tot}nM_{k,r}\). 
More generally, for all \(d\in \Delta^*\), the equality \(fdf=fdf'\) holds in \(\operatorname{tot}nM_{k,r}\).
On the other hand, as \(f'\) is injective, for all \(d\in \Delta\), there is \(d'\in \Delta^*\) such that \(d=f'd'\).

For each \(s\in \Sigma\cup \{f,g\}\), let \(d_s\in \Delta^*\) be such that \(fd_sg\) is equal to \(gs\) in \(\operatorname{tot}nM_{k,r}\). 
Let \(P\) be the monoid with presentation
\[P=\makepres{ \Sigma\cup \{f,g\}}{the relations from \(R\) hold and 
\(fg=\varepsilon\),\\ for all \(s\in \Sigma\cup \{f,g\}\) we have \(gs=fd_s g\),\\ for all \(d\in \Delta\) we have \(d=f'd'\), \(fd'f=fd'f'\) and \(f'df=fdf'\)}.\]
We show that \(P\) is a presentation for \(\operatorname{tot}nM_{k,r}\).
Let \(\psi:(\Sigma\cup \{f,g\})^*\to \operatorname{tot}nM_{k,r}\) be the homomorphism which maps \(f\) to \(f\), \(g\) to \(g\) and each element of \(\Sigma\) to the corresponding element of \(M_{\{\mathbf{0}\}\times \mathfrak{C}_{n,k}}\).
We need to show that \(\psi\) is surjective and that the kernel of \(\phi\) is generated as a congruence by the relations from \(P\).

 We have already shown that all the relations in the presentation \(P\) hold in \(\operatorname{tot}nM_{k,r}\). It follows that the kernel of \(\psi\) includes the relations of \(P\).
Thus we have a well-defined homeomorphis \(\psi:P\to \operatorname{tot}nM_{k,r}\).
Moreover, \(\im(\psi)\supseteq (f  \Delta^*  g)\psi = f  M_{\im(f)}  g=(M_{\im(f)})\phi\).
So \(\psi\) is surjective.

As \(\langle \Sigma |R\rangle\) presents \(M_{\{\mathbf{0}\}\times \mathfrak{C}_{n,k}}\), it follows that the kernel of \(\psi|_{\Sigma^*}\) is generated by \(R\). 
As \(\phi\) is injective, it follows that the kernel of \(\psi|_{\Sigma^*}\cdot \phi\) is also generated by \(R\). 
Thus if \(s,t\in \Sigma^*\) and \((fsg)\psi=(ftg)\psi\), then \((s)\psi|_{\Sigma^*} \phi=f(s)\psi g=f(t)\psi g=(t)\psi|_{\Sigma^*} \phi\) and so \(s\) and \(t\) are related by \(P\).
So \(s,t\in \Sigma^*\) and \((fsg)\psi=(ftg)\psi\), then \(fsg\) and \(ftg\) are related by \(P\). It follows that the kernel of \(\psi|_{f\Sigma^* g}\) is contained in the congruence on the free monoid defined by \(P\).
In particular, whenever two strings are equal in \(P\) to elements of \(f \Delta ^* g\) and have the same image under \(\psi\), then they are equal in \(P\).

It therefore suffices to show that every string in \((\Sigma\cup \{f,g\})^*\) is equal in \(P\) to an element of \(f \Delta ^* g\). From the relation \(fg=\varepsilon\), the identity element has this form. 
It thus suffices to show that the set of elements of $(\Sigma\cup \{f,g\})^*$ equal in \(P\) to strings from \(f  \Delta^* g\) is a right ideal.
Let \(s\in \Sigma \cup \{f,g\}\) and \(fd_0d_1\ldots d_{p-1} g\in f  \Delta^* g\) be arbitrary.
We show that \(fd_0d_1\ldots d_{p-1} gs\) is equal in \(P\) to an element of \(f  \Delta^* g\).
\begin{align*}
    fd_0d_1\ldots d_{p-1} gs &=_{P}f d_0d_1\ldots d_{p-1}  f d_s g\\
    &=_{P}f d_0f'd_1'f'd_2'\ldots f'd_{p-1}'  f d_s g\\
    &=_{P}f d_0fd_1'f'd_2'\ldots f'd_{p-1}'  f d_s g\\
    &=_{P}f d_0fd_1'fd_2'\ldots fd_{p-1}'  f d_s g\\
    &=_{P}f d_0fd_1'fd_2'\ldots fd_{p-1}'  f' d_s g\\
    &=_{P}f d_0f'd_1'f'd_2'\ldots f'd_{p-1}'  f' d_s g\\
    &=_{P}f d_0d_1d_2\ldots d_{p-1}  f' d_s g.
\end{align*}
Note that \(d_0d_1d_2\ldots d_{p-1}  f' d_s\in \Delta^*\), so \(fd_0d_1d_2\ldots d_{p-1}  f' d_s g\in f\Delta^* g\) as required.
\end{proof}

\subsection{Explicit presentation for $\operatorname{tot}M_{2,1}$}
We show in \cref{done} that the monoids tot\(nM_{k,r}\) are finitely presented. The presentations are rather technical so we unpack the construction to give an explicit presentation in the case of tot\(M_{2,1}\) which was Thompson's original monoid. Unfortunately the presentation is rather long.
\begin{theorem}\label{explicit}
        The monoid \(\operatorname{Rel}_{2,1}\) is given by the presentation with generating set \(u,v,\pi_L, U\) subject to the following 82 relations:
        \begin{enumerate}
            \item \(u^6 =\varepsilon\), \(v^3=\varepsilon\),
            \item  \(\overline{u^6} =\varepsilon\), \(\overline{v^3}=\varepsilon\),
            \item \(\overline{\overline{x}}=u^4v\overline{x}v^2 u^2\) for \(x\in \{u,v,U,\pi_L\}\),
               \item \(\overline{\overline{\overline{x}}}=\overline{u^4v\overline{x}v^2 u^2}\) for \(x\in \{u,v,U,\pi_L\}\),
            \item \(x=U  ((vu^3)^2\overline{x\pi_L}  (u^3v^2)^2)\overline{x\pi_R}\) for \(x\in \{u,v,U,\pi_L\}\),
             \item \(\overline{x}=\overline{U  ((vu^3)^2\overline{x\pi_L}  (u^3v^2)^2)\overline{x\pi_R}}\) for \(x\in \{u,v,U,\pi_L\}\),
            \item \(xU=U (vu^3)^2\overline{x\pi_L}(u^3v^2)^2 \overline{x\pi_R}\) for \(x\in \{u,v,U,\pi_L\}\),
            \item\(\overline{xU}=\overline{U (vu^3)^2\overline{x\pi_L}(u^3v^2)^2 \overline{x\pi_R}}\) for \(x\in \{u,v,U,\pi_L\}\),
            \item \((vu^3)^2\overline{x}(u^3v^2)^2 \overline{y}=\overline{y}(vu^3)^2\overline{x}(u^3v^2)^2 \) for \(x,y\in \{u,v,U,\pi_L\}\),
            \item \(\overline{(vu^3)^2\overline{x}(u^3v^2)^2 \overline{y}}=\overline{\overline{y}(vu^3)^2\overline{x}(u^3v^2)^2} \) for \(x,y\in \{u,v,U,\pi_L\}\),
            \item \(U\pi_L=\varepsilon\), \(U\pi_R=\varepsilon\), \(u\pi_L\pi_L=\pi_L\pi_R\), \(u\pi_L\pi_R=\pi_L\pi_L\), \(u\pi_R\pi_L=\pi_R\pi_R\pi_R\), \(u\pi_R\pi_R\pi_L=\pi_R\pi_L\), \(u\pi_R\pi_R\pi_R=\pi_R\pi_R\pi_L\), \(v\pi_L\pi_L=\pi_L\pi_L\), \(v\pi_L\pi_R=\pi_R\pi_R\), \(v\pi_R\pi_L=\pi_L\pi_R\), \(v\pi_R\pi_R=\pi_R\pi_L\),
            \item \(\overline{U\pi_L}=\varepsilon\), \(\overline{U\pi_R}=\varepsilon\), \(\overline{u\pi_L\pi_L}=\overline{\pi_L\pi_R}\), \(\overline{u\pi_L\pi_R}=\overline{\pi_L\pi_L}\), \(\overline{u\pi_R\pi_L}=\overline{\pi_R\pi_R\pi_R}\), \(\overline{u\pi_R\pi_R\pi_L}=\overline{\pi_R\pi_L}\), \(\overline{u\pi_R\pi_R\pi_R}=\overline{\pi_R\pi_R\pi_L}\), \(\overline{v\pi_L\pi_L}=\overline{\pi_L\pi_L}\), \(\overline{v\pi_L\pi_R}=\overline{\pi_R\pi_R}\), \(\overline{v\pi_R\pi_L}=\overline{\pi_L\pi_R}\), \(\overline{v\pi_R\pi_R}=\overline{\pi_R\pi_L}\),
        \end{enumerate}
        where \(\pi_R=(u^3v^2)^2\pi_L\) and   the overline operation \(\overline{\cdot}:\{u,v,\pi_L,U\}^*\to \{u,v,\pi_L,U\}^*\) is the homomorphism defined by
        \begin{enumerate}
            \item \(\overline{U}=Uvu^5v^2u^3vu^5v\pi_L\),
            \item \(\overline{\pi_L}={u^3v{u^2v^2u^3vu^5v}u^3}\pi_L\),
            \item \(\overline{u}=v^2{u^3vu^2v^2u^3vu^5v}u^4{v^2uv^2u^3vu^4v^2}u^5{v^2uv^2u^3vu^4v^2}u^3{vu^2v^2u^3vu^5vu^4}v\),
            \item \(\overline{v}=u^3v{u^2v^2u^3vu^5v}u^3v^2u^4vu^3v^2uv^2u^3vu^4v^2u^3\).
        \end{enumerate}
\end{theorem}
\begin{proof}
    We start by building the presentation from \cref{rewritedefined} when \(n=1\) and \(k=2\).     
    In \cref{presentationdef}, \(\Sigma\) is taken to be any finite generating set for \(\operatorname{Rel}_{2,1}\).
    Thus by \cref{lem:gens}, we can take \(\Sigma\) to be any generating set for Thompson's group \(V\) together with the elements \(U\) and \(\pi^{(\mathbf{0},(0))}\). We write \(\pi_L\) for \(\pi^{(\mathbf{0},(0))}\).
    There are many known generating sets for \(V\) (one is constructed in the proof of \cref{lem:pres}). In order to avoid our presentation getting even longer, we pick a generating set for \(V\) of size \(2\).
    One is given in the paper \cite{bleak2017infinite} (the generators are listed in Corollary 5.2 of that paper but there is a typo in the second one which is corrected in their tree pair diagrams and here). Moreover these generators have finite order so form a monoid generating set.
    The two generators are given by
    \[(ab)(c(de))\xrightarrow[]{u} (ba)(e(cd))\quad\text{ and }\quad (ab)(cd)\xrightarrow[]{v} (ad)(bc).\]
    As such, we pick \(\Sigma=\{U,\pi_L,u,v\}\) (this technically doesn't satisfy the requirement for inverses of invertible elements to be in \(\Sigma\) but this will not be an issue as we have a monoid generating set). 
    We next need a presentation \(\langle \Lambda|R_\Lambda\rangle\) for \(P_{2,1}\). 
    In this case \(P_{2,1}\) is generated by the two elements
    \[(ab)\xrightarrow[]{} a \quad \text{and}\quad (ab)\xrightarrow[]{} b.\]
    The left element happens to coincide with the element \(\pi_L\). We denote the right element by \(\pi_R\).
    As such, we can write \(\Lambda=\{\pi_L,\pi_R\}\) and \(R_{\Lambda}=\varnothing\).
    For point 3 of \cref{presentationdef}, we make the unique valid choice for representing elements of \(\{\pi_L,\pi_R\}^*\).
    
    Note that as 
\[(ab)(cd)\xrightarrow[]{u^3}(ba)(cd)\xrightarrow[]{v^2} (bc)(da)\xrightarrow[]{u^3}(cb)(da)\xrightarrow[]{v^2}(cd)(ab),\]
\((u^3v^2)^2\) is the element \(ab\xrightarrow[]{(u^3v^2)^2} ba\).
    In particular, for point 2 of \cref{presentationdef} we can choose \(\pi_R\) to be the string \((u^3v^2)^2\pi_L\). 
    Point 4 of \cref{presentationdef} doesn't require any choices. Point 5 of \cref{presentationdef} requires us to choose for each \(x\in \Sigma=\{U,\pi_L,u,v\}\) a string \((x)\operatorname{SDef}\in \Sigma^*\) representing its deferment to the \(1\)-cone.
   In particular we need
   \[ab\xrightarrow[]{(U)\operatorname{SDef}}a(bb),\quad a(bc)\xrightarrow[]{(\pi_L)\operatorname{SDef}} ab,\]
   \[ a((bc)(d(ef)))\xrightarrow[]{(u)\operatorname{SDef}} a((cb)(f(de))),\quad a((bc)(de))\xrightarrow[]{(v)\operatorname{SDef}} a((be)(cd)).\]
 By performing the following calculations, we see that we can choose \((U)\operatorname{SDef}=\overline{U}\), \((\pi_L)\operatorname{SDef}=\overline{\pi_L}\), \((u)\operatorname{SDef}=\overline{u}\) and \((v)\operatorname{SDef}=\overline{v}\).
\[(ab)(cd)\xrightarrow[]{U} ((ab)(cd))((ab)(cd))\xrightarrow[]{vu^5v^2u^3vu^5v} ((ab)((cd)(cd)))(ab)\xrightarrow[]{\pi_L}(ab)((cd)(cd))\]
\[a(bc)\xrightarrow[]{u^3vu^2v^2u^3vu^5vu^3} (ac)b\xrightarrow[]{\pi_L} ac\]
\[a((bc)(d(ef)))\xrightarrow[]{v^2{u^3vu^2v^2u^3vu^5v}u^4{v^2uv^2u^3vu^4v^2}u^5{v^2uv^2u^3vu^4v^2}u^3{vu^2v^2u^3vu^5vu^4}v} a((cb)(f(de)))\]
\[a((bc)(de))\xrightarrow[]{u^3v{u^2v^2u^3vu^5v}u^3v^2u^4vu^3v^2uv^2u^3vu^4v^2u^3} a((be)(cd))\]

For point 6 of \cref{presentationdef}, recall from above that
\((u^3v^2)^2\) is the element which swaps the prefixes \(0\) and \(1\). 
Thus the map \(x\mapsto (vu^3)^2\cdot (x)\operatorname{SDef} \cdot (u^3v^2)^2\) is the deferment to the \(0\)-cone.
Moreover the element \(v^2u^2\) maps the \(1\)-cone to the \(11\) cone, so the map
\(x\mapsto u^4v\cdot (x)\operatorname{SDef} \cdot v^2u^2\) is the deferment to the \(11\)-cone.
For point 7 of \cref{presentationdef}, the element \(Y\) is
\(a \xrightarrow[]{Y} aa\), which coincides with the generator \(U\).

We can now list the relations from \cref{presentationdef}. The relations in sets \((a)-(e)\) of \cref{presentationdef} are precisely those listed in points \(1,3,5,7,9\) of this Theorem respectively. The relations in set \((f)\) are almost exactly those in point 11 of this theorem. 
We note that \(N_u=3\) however we do not need all words of length \(3\), as for example the relations \(u\pi_L\pi_L\pi_L=u\pi_L\pi_R\pi_L\) and \(u\pi_L\pi_L\pi_R=u\pi_L\pi_R\pi_R\) are both implied by \(u\pi_L\pi_L=u\pi_L\pi_R\). 
As such we make these simplifications in a few places where possible.
The set \(R_\Lambda\) is empty in this case and the remaining relations are those in the even points of this theorem.
\end{proof}
\begin{corollary}\label{explictpresfinal}
     The monoid tot\(M_{2,1}\) is given by the presentation with generating set \(u,v,\pi_L, U\) subject to the following 82 relations:
        \begin{enumerate}
            \item \(u^6 =\varepsilon\), \(v^3=\varepsilon\),
            \item  \(\overline{u^6} =\varepsilon\), \(\overline{v^3}=\varepsilon\),
            \item \(\overline{\overline{x}}=u^2v^2\overline{x}vu^4\) for \(x\in \{u,v,U,\pi_L\}\),
               \item \(\overline{\overline{\overline{x}}}=\overline{u^2v^2\overline{x}v u^4}\) for \(x\in \{u,v,U,\pi_L\}\),
            \item \(x=\overline{\pi_R x}(v^2u^3)^2\overline{\pi_Lx}(u^3v)^2U\) for \(x\in \{u,v,U,\pi_L\}\),
             \item \(\overline{x}=\overline{\overline{\pi_R x}(v^2u^3)^2\overline{\pi_Lx}(u^3v)^2U}\) for \(x\in \{u,v,U,\pi_L\}\),
            \item \(Ux=\overline{\pi_R x}(v^2u^3)\overline{\pi_L x}(u^3v)^2U\) for \(x\in \{u,v,U,\pi_L\}\),
            \item \(\overline{Ux}=\overline{\overline{\pi_R x}(v^2u^3)\overline{\pi_L x}(u^3v)^2U}\) for \(x\in \{u,v,U,\pi_L\}\),
            \item \(\overline{y}(v^2u^3)^2\overline{x}(u^3v)^2=(v^2u^3)^2\overline{x}(u^3v)^2\overline{y}\) for \(x,y\in \{u,v,U,\pi_L\}\),
            \item \(\overline{\overline{y}(v^2u^3)^2\overline{x}(u^3v)^2}=\overline{(v^2u^3)^2\overline{x}(u^3v)^2\overline{y}}\) for \(x,y\in \{u,v,U,\pi_L\}\),
            \item \(\pi_LU=\varepsilon\), \(\pi_RU=\varepsilon\), \(\pi_L\pi_Lu=\pi_R\pi_L\), \(\pi_R\pi_L u=\pi_L\pi_L\) , \(\pi_L\pi_R u=\pi_R\pi_R\pi_R\), \(\pi_L\pi_R\pi_R u=\pi_L\pi_R\), \(\pi_R\pi_R\pi_R u=\pi_L\pi_R\pi_R\), 
            \(\pi_L\pi_L v=\pi_L\pi_L\), \(\pi_R\pi_L v=\pi_R\pi_R\), \(\pi_L\pi_R v=\pi_R\pi_L\), \(\pi_R\pi_R v=\pi_L\pi_R\),
            \item \(\overline{\pi_L U}=\varepsilon\), \(\overline{\pi_R U}=\varepsilon\), \(\overline{\pi_L\pi_L u}=\overline{\pi_R\pi_L}\), \(\overline{\pi_R\pi_L u}=\overline{\pi_L\pi_L}\), \(\overline{\pi_L\pi_R u}=\overline{\pi_R\pi_R\pi_R}\), \(\overline{\pi_L\pi_R\pi_R u}=\overline{\pi_L\pi_R}\), \(\overline{\pi_R\pi_R\pi_R u}=\overline{\pi_L\pi_R\pi_R}\), \(\overline{\pi_L\pi_L v}=\overline{\pi_L\pi_L}\), \(\overline{\pi_R\pi_L v}=\overline{\pi_R\pi_R}\), \(\overline{\pi_L\pi_R v}=\overline{\pi_R\pi_L}\), \(\overline{\pi_R\pi_R v}=\overline{\pi_L\pi_R}\),
        \end{enumerate}
        where \(\pi_R=\pi_L(v^2u^3)^2\) and the overline operation \(\overline{\cdot}:\{u,v,\pi_L,U\}^*\to \{u,v,\pi_L,U\}^*\) is the homomorphism defined by
        \begin{enumerate}
            \item \(\overline{U}=\pi_Lvu^5vu^3v^2u^5vU\),
            \item \(\overline{\pi_L}=\pi_L u^3vu^5vu^3v^2u^2vu^3\),
            \item \(\overline{u}=vu^4vu^5vu^3v^2u^2vu^3v^2u^4vu^3v^2uv^2u^5v^2u^4vu^3v^2uv^2uv^2u^4vu^5vu^3v^2u^2vu^3v^2\),
            \item \(\overline{v}=u^3v^2u^4vu^3v^2uv^2u^3vu^4v^2u^3vu^5vu^3v^2u^2vu^3\).
        \end{enumerate}    
\end{corollary}
\begin{proof}
        By \cref{lem:pres} and \cref{psidefined}, we can obtain a presentation for tot\(M_{2,1}\) by reversing the relations of the presentation from \cref{explicit}.
\end{proof}

\section{Closing Comments and Questions}

There are a few small issues with extending the work here to \(M_{k,r}\). For an algebra \(A\), we write \(\PEnd(A)\) for the monoid of homomorphisms between subalgebras of \(A\) (considered as partial functions on \(A\)).

Note that \(\mb{F}_{n,k,1}\) has uncountably many subalgebras for all \(n\geq 1\) and \(k\geq 2\).
In particular \(\PEnd(\mb{F}_{n,k,1})\) is uncountable, and hence \(\PEnd(\mb{F}_{n,k,1})\not\cong nM_{k,1}\) as defined by Birget in \cite{birget2020monoid}.

Birget's monoids are originally  defined in terms of finitely generated right ideals of \(\{0,1,\dots,k-1\}^\ast\) which corresponds to finitely generated subalgebras of \(\mb{F}_{1,k,1}\).
The uncountability of \(\PEnd(\mb{F}_{n,k,1})\) comes from the fact that it allows endomorphisms on non-finitely generated subalgebras, so perhaps restricting \(\PEnd\) would yield isomorphisms.
\begin{question}
    Is the family of partial endomorphisms of \(\mathbb{F}_{1,2,1}\) with finitely generated domain a monoid? If so, is this monoid isomorphic to \(M_{2,1}\)?
\end{question}
Birget notes in \cite{birget2020monoid} that it is an open question whether the partial monoids \(nM_{k,1}\) are finitely presented. The methods presented in this paper may be modifiable to work with these structures, as well as the more general \(nM_{k,r}\).
\begin{question}
    Is the monoid \(M_{2,1}\) finitely presented?
\end{question}

\bibliography{biblio}{}

@article{birget2009monoid,
  title={Monoid generalizations of the {Richard Thompson} groups},
  author={Birget, JC},
  journal={Journal of pure and applied algebra},
  volume={213},
  number={2},
  pages={264--278},
  year={2009},
  publisher={Elsevier}
}

@phdthesis{elliott2021constructing,
  author       = {Elliott, L},
  title        = {On constructing topology from algebra},
  school       = {University of St Andrews},
  year         = {2021},
  address      = {St Andrews, Scotland},
  month        = {December},
  url          = {https://le27.github.io/Luke-Elliott/},
}

@article{birget2020monoid,
  title={A monoid version of the {Brin-Higman-Thompson} groups},
  author={Birget, JC},
  journal={arXiv preprint arXiv:2006.15355},
  year={2020}
}

@article{bleak2017infinite,
  author  = {Bleak, C and Quick, M},
  title   = {The infinite simple group $\text{V}$ of {Richard J. Thompson}: presentations by permutations},
  journal = {Groups, Geometry, and Dynamics},
  volume  = {11},
  number  = {4},
  pages   = {1401--1436},
  year    = {2017},
  doi     = {10.4171/GGD/433},
  url     = {https://arxiv.org/abs/1511.02123}
}

@article{jonsson1961two,
  title={On two properties of free algebras},
  author={J{\'o}nsson, B and Tarski, A},
  journal={Mathematica Scandinavica},
  volume={9},
  number={1a},
  pages={95--101},
  year={1961},
  publisher={JSTOR}
}

@inproceedings{dudek1979universal,
  title={On universal algebras having bases of different cardinalities},
  author={Dudek, J},
  booktitle={Colloquium Mathematicum},
  volume={1},
  number={42},
  pages={111--114},
  year={1979}
}

@article{goetz1960bases,
  title={On bases of abstract algebras},
  author={Goetz, A and Ryll-Nardzewski, C},
  journal={Bull. Acad. Polon. Sci. S{\'e}r. Sci. Math. Astronom. Phys},
  volume={8},
  pages={157--161},
  year={1960}
}

@article{higman1974finitely,
  title={Finitely Presented Infinite Simple Groups, Notes on Pure Math. 8, {IAS}},
  author={Higman, G},
  journal={Austral. Nat. Univ., Canberra},
  year={1974}
}

@article{bleak2016further,
  author  = {Bleak, C. and Cameron, P. and Maissel, Y. and Navas, A. and Olukoya, F.},
  title   = {The Further Chameleon Groups of {Richard} {Thompson} and {Graham} {Higman}: {Automorphisms} via Dynamics for the {Higman} Groups ${G_{n,r}}$},
  journal = {Memoirs of the American Mathematical Society},
  volume  = {301},
  number  = {1510},
  year    = {2024},
  doi     = {10.1090/memo/1510}
}

@article{belk2014conjugacy,
  title={Conjugacy and dynamics in {T}hompson’s groups},
  author={Belk, J and Matucci, F},
  journal={Geometriae Dedicata},
  volume={169},
  number={1},
  pages={239--261},
  year={2014},
  publisher={Springer}
}

@article{cannon1996introductory,
  title={$\text{Introductory notes on Richard Thompson’s groups}$},
  author={Cannon, J},
  journal={Enseign. Math.(2)},
  volume={42},
  number={3-4},
  pages={215},
  year={1996}
}

@inproceedings{lawson2021polycyclic,
  title={The polycyclic inverse monoids and the {Thompson groups} revisited},
  author={Lawson, M},
  booktitle={International Conference on Semigroups and Applications},
  pages={179--214},
  year={2021},
  organization={Springer}
}

@article{martinez2013bredon,
  title={Bredon cohomological finiteness conditions for generalisations of {T}hompson groups},
  author={Mart{\'\i}nez-P{\'e}rez, C and Nucinkis, B},
  journal={Groups, Geometry, and Dynamics},
  volume={7},
  number={4},
  pages={931--959},
  year={2013}
}

\bibliographystyle{plain}

\end{document}